\documentclass[a4paper,12pt]{article}

\usepackage{amsthm,amsmath,stmaryrd,bbm,hyperref,geometry,color,authblk}
\usepackage[utf8]{inputenc}
\usepackage{amssymb}
\usepackage{graphicx}
\usepackage{amsfonts,amssymb}
\usepackage{verbatim}
\usepackage{enumitem}
\usepackage[dvipsnames]{xcolor}
\newcommand{\po}{\left(}
\newcommand{\pf}{\right)}
\newcommand{\co}{\left[}
\newcommand{\cf}{\right]}
\newcommand{\cco}{\llbracket}
\newcommand{\ccf}{\rrbracket}
\newcommand{\R}{\mathbb R} 
\newcommand{\T}{\mathbb T} 
\newcommand{\Z}{\mathbb Z} 
\newcommand{\N}{\mathbb N} 
\newcommand{\dd}{\text{d}}
\newcommand{\bX}{\mathbf{X}}

\newcommand{\bx}{\mathbf{x}}

\newcommand{\na}{\nabla}
\newcommand{\1}{\mathbbm{1}} 

\newcommand{\hF}{\widehat{\mathcal F}}
\newcommand{\bmu}{\bar\mu_\theta}
\newcommand{\rp}{\psi}
\newcommand{\C}{\mathcal C}

\newcommand{\re}[1]{#1}
\newcommand{\new}[1]{#1}

\newtheorem{theorem}{Theorem}
\newtheorem{assumption}{Assumption}
\newtheorem{lemma}[theorem]{Lemma}

\newtheorem{proposition}[theorem]{Proposition}
\newtheorem{remark}{Remark}
\newtheorem{example}{Example}

\title{Equivalence between $N$-particle log-Sobolev inequalities and non-linear Łojasiewicz inequalities for a  class of mean field systems}
\author[1,2]{Pierre Monmarché\thanks{pierre.monmarche@univ-eiffel.fr}}
\affil[1]{LAMA, Université Gustave Eiffel, France}
\affil[2]{Institut Universitaire de France}

\begin{document}

\maketitle

\begin{abstract}
The  convergence rate of a free energy Wasserstein gradient flow is quantified by its so-called Polyak-Łojasiewicz (PŁ) constant $\lambda$, which relates the objective function to its dissipation along the flow. Such a  flow is the mean-field limit as $N$ goes to infinity of a system of $N$ interacting particles, whose convergence rate in relative entropy towards its Gibbs measure is quantified by its log-Sobolev constant $\lambda_N$. Different behaviours of $\lambda_N$ as $N$ goes to infinity  thus describe drastically different phenomena, such as fast relaxation or metastability, with many models undergoing phase transitions between these regimes, depending typically on temperature.  Under fairly general conditions, a uniform-in-$N$ log-Sobolev constant (i.e. fast exponential convergence for the particle system) is known to induce a positive PŁ constant (i.e. exponential convergence for the mean-field flow). A conjecture was stated by Delgadino, Gvalani, Pavliotis and Smith according to which the converse implication was true ($\lambda>0$ implies $\liminf \lambda_N >0$), even with $\lim \lambda_N =\lambda$.
 
 First, we will prove this converse implication, although without the equality $\lim \lambda_N = \lambda$, for a general class of mean-field models. Second, we also notice that this implication fails if the free energy minimiser is not unique, and provide an explicit counter-example.   Third, we also consider the same question of relating $N$-particle and mean-field inequalities in the context of more general Łojasiewicz inequalities, which correspond to polynomial (instead of exponential) convergence rates, and can describe the situation exactly at a phase transition.

\end{abstract}

\section{Introduction}

\subsection{Free energy minimization and mean-field Gibbs sampling}

For either $E=\R^d$ or $E=\T^d$ with $\T=\R/\Z$, denote by  $\mathcal P_2(E)$ the space of probability measures  on $\mu$ (with a finite second-order moment if $E=\R^d$). Consider some energy $\mathcal E:\mathcal P_2(E) \rightarrow (-\infty,\infty]$. A classical example in statistical physics is
\begin{equation}\label{eq:EVW}
    \mathcal E(\mu) = \int_E V(x)\mu(\dd x) + \frac12 \int_{E^2} W(x,y) \mu(\dd x)\mu(\dd y)\,,
\end{equation}
for some external potential $V:E\rightarrow \R$ and pairwise interaction potential $W:E^2\rightarrow \R$.  The  free energy (at temperature normalised to $1$) associated to $\mathcal E$ is defined as
\begin{equation}
\label{eq:defF}
\mathcal F(\mu) = \mathcal E(\mu) + \mathcal H(\mu)\,,
\end{equation}
where   $\mathcal H(\mu) = \int_{E} \mu \ln\mu$ stands for the entropy (with $\mathcal H(\mu)=+\infty$ if $\mu$ doesn't have a density).

For $N\geqslant 1$, writing
\[\pi_{\bx} = \frac1N\sum_{i=1}^N \delta_{x_i}\]
the empirical distribution of $\bx=(x_1,\dots,x_N)\in E^N$,  the $N$-particle Gibbs measure associated to $\mathcal E$ is the probability measure with density
\begin{equation}
\label{eq:defmuinftyN}
\mu_\infty^N(\bx) \propto \exp\po - U_N(\bx)\pf,\qquad U_N(\bx) = N \mathcal E \po \pi_{\bx}\pf\,, 
\end{equation}
assuming that $e^{-U_N}\in L^1$.

The problem of minimizing a free energy of the form~\eqref{eq:defF} arises in many fields. In statistical physics or mathematical biology, the mean-field evolution of many models can be described as a gradient descent of such a quantity, as in \cite{bashiri2020gradient,sandier2004gamma,bunne2022proximal}. In this situation, the entropy term accounts for microscopic stochastic effects (e.g. thermal agitation), and the Gibbs measure is the equilibrium density of the corresponding microscopic (or individual-based) model of $N$ particles.  This optimisation problem also arises for numerical reasons, for instance for high-dimensional machine learning algorithms~\cite{Szpruch,mei2018mean,geshkovski2025mathematical}, optimal transport~\cite{arbel2019maximum}, variational inference~\cite{arbel2019maximum,lambert2022variational}, optimization~\cite{chizat,peyre2015entropic} or sampling~\cite{wibisono2018sampling,lelievre2026convergence}. In this situation, the entropy is often a penalisation term added to enforce minimisers to be smooth and to improve the stability of the algorithms, and the Gibbs measure is the equilibrium of practical algorithms which are implemented with a finite number of particles.

Following the theory of gradient flows in metric spaces, the gradient flow of~\eqref{eq:defF} with respect to the $L^2$ Wasserstein distance $\mathcal W_2$ is
\begin{equation}
\label{eq:EDP_gradient_flow}
\partial_t \rho_t = \na\cdot \po \rho_t D\mathcal E(\rho_t,\cdot) \pf +\Delta\rho_t\,,
\end{equation}
with $D\mathcal E(\mu,x) = \na_x \frac{\delta \mathcal E}{\delta \mu}(\mu,x)$, where $\frac{\delta \mathcal E}{\delta \mu}$ is the first variation (or Fréchet derivative) of $\mathcal E$, see \cite{ambrosio2005gradient} for definitions and details. In all this work we assume that $D\mathcal E(\mu,\cdot)$ exists and is continuous for all $\mu\in\mathcal P_2(E)$. Under suitable conditions, it is known that $\rho_t$ is the mean-field limit as $N\rightarrow \infty$ of $\pi_{\bX_t}$ where the particle system $\bX_t=(X_t^1,\dots,X_t^N)\in E^N$ solves
\[\forall i\in\cco 1,N\ccf,\qquad \dd X_t^i = -D\mathcal E(\pi_{\bX_t},X_t^i) \dd t + \sqrt{2}\dd B_t^i\,,\]
where $\mathbf{B}=(B^1,\dots,B^N)$ are independent $d$-dimensional Brownian motions. This can be equivalently written as
\begin{equation}
\label{eq:EDS_Langevin}
\dd \bX_t = - \na U_N(\bX_t) \dd t + \sqrt{2}\dd \mathbf{B}_t\,.
\end{equation}
In other words, $\bX$ is a Langevin process, reversible with respect to $\mu_\infty^N$. Under mild conditions, it is ergodic and $\bX_t$ converges in law to $\mu_\infty^N$ as $t\rightarrow \infty$ (independently from the initial distribution).  The law $\rho_t^N$ of $\bX_t$ solves the Fokker-Planck equation
\begin{equation}\label{eq:EDP_N_particles}
\partial_t \rho_t^N = \na \cdot \po \rho_t^N \na U_N \pf + \Delta \rho_t^N\,. 
\end{equation}
Notice that this is of the form~\eqref{eq:EDP_gradient_flow}, i.e. $\rho_t^N$ follows the Wasserstein gradient flow in $\mathcal P_2(E^N)$ of the $N$-particle free energy $\int_{E^N} U_N \rho^N + \mathcal H(\rho_t^N) $. Up to an additive constant (which doesn't affect the gradient flow), this $N$-particle free energy is equal to the relative entropy $\mathcal H(\rho^N|\mu_\infty^N) = \int_{E^N} \rho_t^N \ln \frac{\rho_t^N}{\mu_\infty^N}$.

\subsection{Convergence rates}

Assume that $\mathcal F$ is lower bounded and write $\mathcal F^c=\mathcal F - \inf \mathcal F$. The dissipation of the free energy along the flow~\eqref{eq:EDP_gradient_flow} is
\begin{equation}
\label{eq:dissipation}
-\partial_t \mathcal F(\rho_t) = \int_{E} \left|\na \ln \rho_t + D\mathcal E(\rho_t,\cdot)\right|^2 \rho_t =:  \mathcal I(\rho_t)\,.
\end{equation}
As a consequence, the convergence rate of the gradient flow to the set of minimizers can be quantified in terms of the  Polyak-Łojasiewicz (PŁ) constant defined as
\begin{equation}
\label{eq:deflambdaPL}
\lambda = \inf \left\{\frac{\mathcal I(\rho)}{2 \mathcal F^c(\rho)},\, \rho \in  \mathcal P_2(E),\, \mathcal F^c(\rho)\neq 0\right\}\,.
\end{equation}
Indeed, due to~\eqref{eq:dissipation}, $\lambda$ is the largest constant such that
\[\forall t\geqslant 0,\ \forall \rho_0 \in\mathcal P_2(E),\qquad \mathcal F^c(\rho_t) \leqslant e^{-2\lambda t} \mathcal F^c(\rho_0)\,.\]
We say that $\mathcal F$ satisfies a PŁ inequality (with constant $\lambda$) if $\lambda>0$. This condition clearly implies that all critical points (i.e. all the stationary solutions of the gradient flow~\eqref{eq:EDP_gradient_flow}, for which $\mathcal I(\rho)=0$) are  global minimisers, and that the set of global minimisers is connected (otherwise a path of minimal height between two unconnected minimisers have to cross a saddle point).

For the particle system, corresponding to~\eqref{eq:dissipation}, the dissipation of the relative entropy along~\eqref{eq:EDP_N_particles} is given by the Fisher information:
\[-\partial_t \mathcal H\po \rho_t^N|\mu_\infty^N \pf = \int_{E^N} \left|\na \ln \frac{\rho_t^N}{\mu_\infty^N}\right|^2\rho_t^N =: \mathcal I(\rho_t^N|\mu_\infty^N)\,.\]
In this context, the PŁ constant
\begin{equation}
\label{eq:deflambdaNLSI}
\lambda_N = \inf \left\{\frac{\mathcal I(\rho|\mu_\infty^N)}{2 \mathcal H(\rho|\mu_\infty^N)},\, \rho \in  \mathcal P_2(E^N)\setminus\{\mu_\infty^N\}\right\}
\end{equation}
is more classically known as the log-Sobolev constant of $\mu_\infty^N$.  It is the largest constant such that
\[\forall t\geqslant 0,\ \forall \rho_0^N \in\mathcal P_2(E),\qquad \mathcal H(\rho_t^N|\mu_\infty^N) \leqslant e^{-2\lambda_N t} \mathcal H(\rho_0^N|\mu_\infty^N)\,.\]
We say that $\mu_\infty^N$ satisfies a log-Sobolev inequality (LSI) if $\lambda_N>0$. Moreover, we say that the LSI is uniform in $N$ if
\begin{equation}
\label{eq:LSI-unifN}
\underline{\lambda}:=  \liminf_{N\rightarrow \infty} \lambda_N > 0\,.
\end{equation}

\subsection{Motivation and objective}\label{sec:motivation}

Whether the LSI is uniform in $N$ or not (and, if not, how $\lambda_N$ vanishes with $N$) has important conceptual and practical consequences. In many cases where $\lambda_N$ goes to zero, actually, there is some $c>0$ such that $\lambda_N =\mathcal O( e^{-cN})$. This corresponds to  metastable behaviours of the particle system: relaxation to equilibrium requires rare transitions which occur at a time-scale of order $e^{cN}$ (and are not seen anymore for the mean-field limit~\eqref{eq:EDP_gradient_flow}). Rare events occurring at an exponential time-scale are in the same  family as random monkeys writing Hamlet: we don't see them in practice. From a modelling point of view, it means that what is observed in practice is not described by the relaxation to the Gibbs measure, and from an algorithmic point of view, it means that the simulation time is prohibitive (actually, beating metastability for sampling has been a major issue for decades, see e.g. \cite{lelievre2012two,lelievre2010free} and references within, and see \cite{MonmarcheLeclerc} in the mean-field situation considered here).

In many models (but not always, see \cite[Remark 1]{Mtoymodel} for a counter-example), the LSI can be shown to be uniform in $N$ for large temperature (i.e. for the free energy $\mathcal E + \sigma^2 \mathcal H$ with a large $\sigma^2$), in particular thanks to the strong convexity of the entropy. Some  interesting stories occur in situations where we start decreasing $\sigma^2$ and, reaching some critical value $\sigma_c^2>0$, the uniformity of the LSI vanishes. For instance, for the mean-field symmetric double-well Curie-Weiss model studied in \cite{Dagallier,Monmarchemetastable,Mtoymodel}, there is a critical temperature $\sigma_c^2>0$ such that, for $\sigma^2>\sigma_c^2$, the LSI is uniform in $N$, for $\sigma^2=\sigma_c^2$, $\lambda_N$ is of order $N^{-1/2}$, and for $\sigma^2<\sigma_c^2$, $\lambda_N=\mathcal O(e^{-cN})$ for some $c>0$. This brutal change of behavior when some parameter (here, the temperature) passes through a critical value is called a phase transition. For instance for the mean-field Curie-Weiss model this gives an explanation in simplified settings of the brutal change of magnetic properties of some materials at a given temperature.

Now, for the same Curie-Weiss model, at the same critical temperature, the deterministic mean-field flow~\eqref{eq:EDP_gradient_flow} also undergoes a phase transition: for $\sigma^2>\sigma_c^2$, $\mathcal F$ has a unique global minimiser and satisfies a PŁ inequality; for $\sigma^2=\sigma_c^2$, there is still a unique global minimiser but $\lambda=0$; for $\sigma^2 <\sigma_c^2$, there are two global minimizers and one saddle point.

The fact that the critical temperature is the same for the mean-field flow~\eqref{eq:EDP_gradient_flow} and the $N$-particle system~\eqref{eq:EDP_N_particles} is equivalent to say that, in this model,
\begin{equation}
\label{eq:equivalence}
\lambda>0 \qquad \Leftrightarrow\qquad \liminf_{N\rightarrow \infty} \lambda_N >0\,.
\end{equation}
The question is thus to know if this is actually a general fact. If this equivalence is true, it means that if we design an algorithm in the idealized mean-field form~\eqref{eq:EDP_gradient_flow} (as this is often the case) and it has good convergence properties in the sense that it satisfies a PŁ inequality, then we know that the particle system that we are going to use in practice will also have good convergence properties (for end-to-end guarantees between  feasible time-discrete particle schemes and the minimiser of $\mathcal F$ under a uniform-in-$N$ LSI, we refer to \cite{Suzukietal,MonmarcheSchuh}). This is interesting  since it is  often  a priori simpler to prove a single PŁ inequality (see the criteria in \cite[Sections 2.2 and 3]{monmarche2025local}) than a uniform family of LSI depending on $N$. For instance, PŁ inequalities under a flat-convexity condition were proven in~\cite{chizat,nitanda2022convex} before and with much simpler proofs than uniform LSI under a similar condition~\cite{SongboLSI,Chewietal,kook2024sampling}.

In fact, one of the two implications in~\eqref{eq:equivalence} is known:  it has been observed in \cite{Pavliotis,GuillinWuZhang} that, under suitable general conditions (see Theorem~\ref{thm:unifLSI->PL} below for instance),
\begin{equation}
\label{eq:implication_classique}
\lambda \geqslant \limsup_{N\rightarrow \infty} \lambda_N\,.
\end{equation}
This observation has led Delgadino, Gvalani, Pavliotis and Smith to propose in \cite{Pavliotis} the following conjecture, stronger than~\eqref{eq:equivalence}:
\[\lambda = \lim_{N\rightarrow \infty} \lambda_N\,.\]
However, for now, this has only been proven in very simple toy models, as in \cite[Proposition 4]{Mtoymodel}. Even the weaker equivalence~\eqref{eq:equivalence}, which would be very useful as discussed above, has remained open.

\subsection{Overview of our results and organisation}

First, our main contribution is the proof of the equivalence~\eqref{eq:equivalence} for a general class of models, as stated in Theorem~\ref{thm:final}. The proof involves several intermediary steps of interest for themselves, which are gathered in Section~\ref{sec:intermediaryresults}. We also give a criterion for uniform-in-$N$ LSI under a convexity condition stronger than the PŁ inequality in Theorem~\ref{thm:generalconvexity} (where, with the vocabulary introduced in Section~\ref{sec:parametrized}, the dimension of the parameter can be infinite).

Second,  noticing with Theorem~\ref{thm:final}  that the equivalence~\eqref{eq:equivalence}  requires $\mathcal F$ to have a unique global minimizer, which is not implied by the fact $\lambda>0$, we give  a simple explicit counter-example to~\eqref{eq:equivalence} in Proposition~\ref{prop:counterexample}.

Third, we will conduct a similar study in degenerate cases where the PŁ inequality is replaced by a general Łojasiewicz inequality of the form
\begin{equation}
\label{eq:Łojasiewicz_degenere}
\forall \rho \in \mathcal P_2(E),\qquad \mathcal F^c(\rho) \leqslant \Phi\po \mathcal I(\rho)\pf\,,
\end{equation}
where $\Phi:\R_+\rightarrow \R_+$ is a non-decreasing function. The PŁ inequality corresponds to the linear case $\Phi(r) = \frac{r}{2\lambda}$. Similarly to the optimal PŁ constant~\eqref{eq:deflambdaPL}, the optimal non-decreasing $\Phi$ is the Łojasiewicz profile defined as
\begin{equation}\label{eq:defPhiprofile}
\Phi(r) := \sup\left\{ \mathcal F^c(\rho)\, : \, \mathcal I(\rho)\leqslant r\right\}\,.
\end{equation}
As will be clear, the relevant notion of uniform-in-$N$ general Łojasiewicz LSI in that case is
\begin{equation}
\label{eq:LSI_degenere}
\forall \rho^N \in \mathcal P_2(E^N),\qquad \frac1N \mathcal H(\rho^N|\mu_\infty^N) \leqslant \Phi_N\po \frac1N \mathcal I(\rho^N|\mu_\infty^N)\pf\,,
\end{equation}
with a family of non-decreasing functions $\Phi_N$ with
\[\widetilde{\Phi}(r) := \limsup_{N\rightarrow \infty} \Phi_N(r) < \infty \qquad \forall r\geqslant 0\,.\] 
This reduces to standard uniform-in-$N$ LSI when $\Phi_N$ is linear for all $N\geqslant 1$. The optimal non-decreasing $\Phi_N$ in~\eqref{eq:LSI_degenere} corresponds to~\eqref{eq:defPhiprofile} applied to the gradient flow of $\frac1N \mathcal H(\cdot|\mu_\infty^N)$.

This degenerate situation is interesting in particular since it can occur at phase transitions. For instance, for the mean-field symmetric double-well Curie-Weiss model, as stated in \cite[Theorem 29]{Mtoymodel}, at $\sigma^2=\sigma_c^2$, both these inequalities are satisfies with $\Phi(r)$ and $\widetilde{\Phi}(r)$ of the form $C\min(r,r^{2/3})$ for some constant $C>0$. It provides polynomial convergence rates, uniformly in $N$ for the particle system.

\medskip

The rest of this work is organised as follows. In the next section, we conclude this introduction with a discussion about the most related existing literature. Our main results in the non-degenerate case, including Theorem~\ref{thm:final} for the equivalence~\eqref{eq:equivalence} and Theorem~\ref{thm:generalconvexity} for a uniform LSI criterion under a convexity condition, are gathered in Section~\ref{sec:results}. A counter-example to~\eqref{eq:equivalence} is presented in Section~\ref{sec:contreexemple}, leading to Proposition~\ref{prop:counterexample}. Section~\ref{sec:degenerate} is devoted to the degenerate case, with all results gathered in Theorem~\ref{thm:intermediaryimplications-dege}. The proof for the uniform LSI is then given in Section~\ref{sec:proofLSIN}, while the other proofs are gathered in Section~\ref{sec:otherproofs}.  

\subsection{Related works}

The question of establishing uniform-in-$N$ LSI, which has been for decades the topic of many works in various settings, has recently made substantial progresses in the mean-field case, motivated in particular by the applications for high-dimensional algorithms. Before that, the methods were  mainly restricted to strongly displacement-convex cases (using the Bakry-Emery criterion for LSI) or small perturbations of the i.i.d. case $\mathcal E(\rho) = \int_E V\rho$ (mimicking the proof of LSI tensorization), see \cite{Malrieu,GuillinWuZhang,Pavliotis,CMCV}. This only covered cases with no phase transition or only in the very high temperature regime (far from criticality). The two-scale approach  introduced in~\cite{bauerschmidt2019very} somehow combines the two approaches: the i.i.d. case applies to microscopic scales, while convexity is used for the macroscopic scales where a low-temperature LSI is required (see the definitions in Section~\ref{sec:proofLSIN} for details).

More recently, the flat-convexity of $\mathcal E$ (i.e. convexity of $t\mapsto \mathcal E(t\rho + (1-t)\mu)$, by contrast to displacement convexity which is convexity along Wasserstein geodesics) has emerged as an important condition for this problem. It was first used to prove mean-field PŁ inequalities~\cite{chizat,nitanda2022convex} (i.e. $\lambda>0$) and then uniform LSI under flat-convexity (and some other conditions) was established in \cite{SongboLSI,Chewietal,kook2024sampling}. This is a generalization of the i.i.d. case, since $\rho \mapsto \int_E V\rho$ is linear, hence flat-convex.  The proof in \cite{SongboLSI} can be understood as a  generalisation of the proof for the tensorization property of LSI. This flat-convexity condition is very different from displacement-convexity: for instance, $\rho \mapsto \int_E V\rho$ is displacement-convex if and only $V$ is convex. Many interesting algorithms are flat convex but far from displacement-convex, see~\cite{chizat,nitanda2022convex} and references within. 

In parallel, for low-temperature LSI, the usual strong convexity condition was replaced in \cite{chewi2024ballistic,Monmarchemetastable}  by weaker PL-type conditions.

Then,~\cite{Dagallier,Monmarchemetastable} combined these two type of results inside the two-scale approach of~\cite{bauerschmidt2019very}, relying on flat-convexity for the microscopic scale and a PŁ condition for the macroscopic one. This yields uniform-in-$N$ LSI under more general conditions than previous works, and in particular in \cite{Dagallier} the uniform LSI is proven up to criticality (i.e. for any $\sigma^2>\sigma_c^2$) for the symmetric double-well Curie-Weiss model. In fact the proof of~\cite{Dagallier} shows that the uniform LSI is implied by a PŁ inequality for a suitable coarse-grained free energy. 

Our approach is to show that the coarse-grained PŁ inequality is implied by the mean-field one (i.e. by $\lambda>0$) and then essentially follow~\cite{Dagallier}. We  discuss further the difference between our work and~\cite{Dagallier} in Section~\ref{sec:PL->LSI}.

\section{Results for the inequality equivalence}\label{sec:results}

We will split the equivalence~\eqref{eq:equivalence} as a sequence of intermediary implications. These intermediary steps can be of interest by themselves, and some them are proven under weaker assumptions than the others or with more explicit constants, which also motivate to state them separately.

This section  is thus organised as follows: first, Section~\ref{sec:parametrized} introduces the class of models considered in this work, where the non-flat-convex part of the energy is written in terms of a parameter map $\varphi$ with value in some Hilbert space $\mathbb H$. Given such a parameter map, we can thus introduce in Section~\ref{sec:coarsegrainedIneq} a two-scale decomposition of the problem, with a macroscopic part in $\mathbb H$ and a microscopic part on $E^N$ modulated by the parameter, which allows to state several coarse-grained PŁ and log-Sobolev inequalities. The results relating all these inequalities are provided in Section~\ref{sec:intermediaryresults}. In Section~\ref{sec:finalresult}, these intermediary results are combined to get Theorem~\ref{thm:final}, which establishes the equivalence~\eqref{eq:equivalence} under suitable conditions.  

\medskip

The following basic conditions are \emph{implicitly} enforced in all the work in order to ensure the well-posedness of the objects considered (for instance the fact that $e^{-U_N}\in L^1(E^N)$ for all $N\geqslant 1$ or that $\exp(-\frac{\delta\mathcal E}{\delta m}(\mu,\cdot) \in L^1(E)$ for all $\mu \in \mathcal P_2(E)$) and that $\mathcal F$ is lower-bounded and admits a global minimizer (as in the proof of Lemma~\ref{lem:uniqueminimumtheta}).

\begin{assumption}\label{ass:basic}
For all $\mu\in\mathcal P_2(E)$, the derivatives $\frac{\delta \mathcal E}{\delta m}(\mu,\cdot) $ and $D\mathcal E(\mu,\cdot)$ exists and are continuous and, if $E=\R^d$, there exist $c_\mu,C_\mu>0$ such that $\frac{\delta \mathcal E}{\delta m}(\mu,x) \geqslant c_\mu  |x|^2  - C_\mu$ for all $x\in \R^d$. Moreover, if $E=\R^d$, there exists $c_0,C_0>0$ such that  $\mathcal E(\mu) \geqslant c_0 \int_{\R^d} |x|^2 \mu(\dd x) - C_0$ for all $\mu \in\mathcal P_2(\R^d)$.
\end{assumption}

\subsection{Parametrized energies}\label{sec:parametrized}

We use the notation $\mu(f)  = \int f \dd \mu$. We consider energies of the form
\begin{equation}\label{loc:Eparameter}
    \mathcal E(\mu) = \mu(V) +  \mathcal E_c(\mu) + R\po \mu(\varphi)\pf\,,
\end{equation}
where $V:E\rightarrow \R$, $\mathcal E_c$ is flat-convex (recall that this means that $t\mapsto \mathcal E_c(t\rho+(1-t)\mu)$ is convex for all $\rho,\mu \in\mathcal P_2(E)$) and, for some separable Hilbert space $\mathbb H$,
\begin{equation}
    \label{loc:phiR}
\varphi:E \rightarrow \mathbb H,\qquad R:\mathbb H\rightarrow \R\,.
\end{equation}
 In other words, we decompose the non-flat-convex  part of the energy as the composition of $\mu \mapsto \mu(\varphi)$ (which is linear) and $R$ which is non-linear (but over a Hilbert space, which is somehow more convenient than non-linear functions over the Wasserstein space). We refer to $\mu \mapsto R(\mu(\varphi))$ as a parametrized energy, with parameter  $\mu(\varphi)$ and parameter map $\varphi$.
 
 As an example, for~\eqref{eq:EVW} on $E=\R^d$ with quadratic pairwise interaction  $W(x,y)=\kappa |x-y|^2$,
\[\mathcal E(\mu) = \int_{E} \po V(x) +  \kappa |x|^2 \pf \mu(\dd x) - \kappa\left|\int_{E} x\mu(\dd x)\right|^2,\]
i.e. we simply have $\mathbb H = \R^d$, $\varphi(x) = x$ and $R(y)=-\kappa |y|^2$. For more general pairwise interactions~\eqref{eq:EVW}, the same decomposition in fact holds with a Mercer or a Fourier decomposition of the potential, of the form
\begin{equation}
\label{eq:modes}
W(x,y) = W_0(x) + W_0(y) + \sum_{k \in \N} m_k(x)m_k(y) - \sum_{k\in\N} n_k(x)n_k(y)\,, 
\end{equation}
setting $\mathcal E_c(\mu) = \sum_{j\in\N} \po \mu(m_k)\pf^2 $, $\varphi(x)=(n_k(x))_{k\in\N} \in \ell^2(\N)$ and $R(\psi)=-\|\psi\|_{\ell^2}^2$.

Beyond pairwise interactions, $R$ can be an higher order polynomial to cover $k$-particle interactions with $k>2$, or a more complicated non-linear map over a Hilbert space of functions or a Reproducing Kernel Hilbert space, as in various algorithmic situations~\cite{Szpruch,arbel2019maximum}. A non-quadratic $R$ also arises in the modified energy method of~\cite{Monmarchemetastable}.  In summary, the form~\eqref{loc:Eparameter} is actually very general.

Moreover, assume that the curvature of $R$ is lower-bounded, in the sense that there exists $\kappa\geqslant 0$  such that $R+\kappa|\cdot|^2$ is convex (we write $\theta\cdot \zeta$ and $|\theta|$ the scalar product and norm for $\theta,\zeta\in \mathbb{H}$). Then, up to replacing $\mathcal E_c$ by $\mathcal E_c(\mu) + R(\mu(\varphi)) + \kappa |\mu(\varphi)|^2$ (which is flat-convex) and $\varphi$ by $\sqrt{2\kappa} \varphi$, the energy $\mathcal E$ can be written in the form
\begin{equation}
    \label{eq:energy-quadra}
\mathcal E(\mu) = \mu(V) +  \mathcal E_c(\mu) - \frac12 |\mu(\varphi)|^2\,.
\end{equation}
In other words, instead of considering~\eqref{loc:Eparameter} with a lower bound on the curvature of $R$, there is no loss of generality to restrict to~\eqref{eq:energy-quadra}, which we do in the following.

For some  results, we will need that $\dim \mathbb H < +\infty$, but the infinite dimensional case is interesting (for instance for a decomposition~\eqref{eq:modes} with an infinity of modes $n_k$ in the concave part). For this reason, our analysis will cover as much as possible this situation. In that purpose, we will need the additional condition that $\varphi = \C \psi$ for some $\psi :E \rightarrow \mathbb H$ where $\C$ is a self-adjoint operator on $E$ such that $\C^2$ is trace-class, i.e. there exists a Hilbert basis $(e_i)_{i\in\N}$ of $\mathbb H$ such that $\C\psi = \sum_{i\in\N} r_i (e_i \cdot \psi) e_i$  for some  $(r_i)_{i\in\N} \in \R^{\N}$ with $\sum_{i\in\N} r_i^2 <\infty$. For instance, if $\mathbb H=\ell^2(\N)$ and $\varphi=(\varphi_k)_{k\in\N}$ is such that, for all $x\in E$, $\sum_{k\in\N} k^2 |\varphi_k(x)|^2 <\infty $, then we can define $\psi_k(x) = k \varphi_k(x) $, and then the operator $\C\theta = (\theta_k/k)_{k\in\N}$ is trace-class. In the case where $\dim \mathbb H$ is finite and fixed, then we can keep in mind the simple situation where $\C=I$, $\varphi=\psi$ (although even in this case, taking $\C\neq I$ might be necessary  to get dimension-free estimates).

\subsection{Coarse-grained inequalities}\label{sec:coarsegrainedIneq}

In the rest of this section, we consider that we are given a separable Hilbert space $\mathbb H$, a map $\rp:E\rightarrow \mathbb H$, a   trace-class positive self-adjoint operator $\C^2$ over $\mathbb H$ and set $\varphi=\C \rp$. Although we don't assume it for now, we have in mind that the energy is of the form~\eqref{eq:energy-quadra} (we will explicitly say when this assumption is needed).  

For $\theta \in \mathbb H$, introduce the coarse-grained free energy 
\begin{equation}\label{eq:defFhat}
\hF(\theta) = \inf\left\{\mathcal F(\rho)\ : \ \rho(\rp)=\theta\right\}\,,
\end{equation}
with $\inf\emptyset = +\infty$, and 
\begin{equation}
\label{eq:defomega1}
\hat \omega(\theta) = \inf_{\rho\in \mathcal P_2(E)} \left\{\mathcal F(\rho) + \frac{1}{2}|\rho(\varphi)|^2 - \rho(\psi)\cdot  \theta \right\} = \inf_{m\in\mathbb H } \left\{\hF(m) + \frac12 |\C m|^2-  m \cdot \theta \right\}\,.  
\end{equation}
\re{This is clearly finite for all $\theta\in\mathbb H$}. For $N\geqslant 1$, writing $\gamma_{\C}$ the centered Gaussian measure on $\mathbb H$ with covariance $N^{-1} \C^2$,  the renormalised measure $\tilde \nu_N$ over $\mathbb H$ is \re{defined as
\begin{equation}
\label{eq:deftildenuN}
\tilde \nu_N(\dd \theta) \propto e^{-N\hat \omega (\theta)} \gamma_{\C}(\dd \theta),
\end{equation}
provided $e^{-N \hat \omega} \in L^1(\gamma_{\C})$ (this will be the case under the conditions enforced below, see Lemma~\ref{lem:omegaNtoomega}). When $\dim\mathbb H<\infty$, $\tilde \nu_N \propto e^{-N\omega}$ is the Gibbs measure at temperature $N^{-1}$ with the so-called renormalised potential
\begin{equation}
\label{eq:defomega}
\omega(\theta) = \hat \omega (\theta) + \frac{1}{2}|\C^{-1} \theta|^2 = \inf_{\rho\in \mathcal P_2(E)} \left\{\mathcal F(\rho) + \frac{1}{2}|\rho(\varphi)  -\C^{-1} \theta|^2 \right\},
\end{equation}
for $\theta \in \mathrm{Im}\C$ (otherwise we set $\omega(\theta) = +\infty = |\C^{-1}\theta|^2$). We informally interpret $\tilde \nu_N$ similarly when $\dim\mathbb H=\infty$.  Large values of $N$ thus corresponds to a low-temperature regime for $\tilde \nu_N$. 
 }

\re{To avoid technical discussions on well-posedness, we will regularly work  under the following condition:
\begin{assumption}\label{assum:finitedim}
The dimension of $\mathbb H$ is finite and $\C$ is non-singular.
\end{assumption}
\begin{remark}\label{rem:diminfini}
Once we assume $\dim \mathbb H<\infty$, assuming that $\C$ is non-singular is not restrictive since in the form~\eqref{eq:energy-quadra} with $\varphi=\C\psi$ we can always reduce the Hilbert space to the image of $\C$. In order to get results when $\dim \mathbb H=+\infty$, it is often sufficient to apply the finite-dimensional results to the case where $\varphi$ is replaced by $\Pi_n\varphi$ where $\Pi_n$ is a projection into a $n$-dimensional subspace and then let $n\rightarrow \infty$. However, to do so, it is crucial that all the estimates we get only depend  on dimension-free parameters. In particular, the trace-class assumption on $\C^2$ is crucial to get  that $\sup_n \mathrm{Tr}(\Pi_n^* \C^2 \Pi_n) \leqslant \mathrm{Tr}(\C^2)<\infty$. Moreover, when taking limits as $N\rightarrow \infty$ as in~\eqref{eq:equivalence}, we have to make sure that this commutes with the limit $n\rightarrow\infty$. For instance the result in~\cite{chewi2024ballistic} gives a result as $N\rightarrow \infty$ but the error terms explicitly depends on the ambient dimension and thus we cannot rely on it in the infinite dimensional case; which is the reason why, eventually, Theorem~\ref{thm:final} is restricted to the finite-dimensional case.
\end{remark}
}

\medskip

The  links between $\mathcal F$, $\hF$ and $\omega$ are clarified in Lemma~\ref{lem:omegaFhat}. For $\theta\in\mathbb H$ and $\rho\in\mathcal P_2(E)$, we write
\begin{equation}
\label{eq:defFtheta}
\mathcal F_\theta(\rho) = \mathcal F(\rho) + \frac12|\rho(\varphi)|^2 - \re{\rho(\rp) \cdot \theta}\,.
\end{equation}
Consider the following conditions:

\begin{assumption}\label{assum:baseLemme1}
For all $\theta \in \mathbb H$, $  \mathcal F_\theta $
admits on $\mathcal P_2(E)$ a unique global minimiser $\bar\mu_\theta \in \mathcal P_2(E)$. Moreover, writing $m_\theta = \bar\mu_\theta(\rp)$,
\begin{equation}
\label{eq:regularitymutheta}
\po \theta \mapsto \mathcal F(\bar\mu_\theta) \pf \in \mathcal C^1(\mathbb H,\R),\qquad \po\theta \mapsto m_\theta\pf\in\mathcal C^1(\mathbb H,\mathbb H)\,.
\end{equation}
\end{assumption}

A simple condition to check that $  \mathcal F_\theta $ has a unique global minimiser for all $\theta$ is given in Lemma~\ref{lem:uniqueminimumtheta}. The regularity condition~\eqref{eq:regularitymutheta} is further discussed in Remark~\ref{rem:regularityPDE}. 

Under Assumption~\ref{assum:baseLemme1}, Lemma~\ref{lem:omegaFhat} shows that  $\hF$ is $\mathcal C^1$ over $\mathcal M:=\{m_\theta,\ \theta\in \mathbb H\}$ and $\omega$ is $\mathcal C^1$ over $\mathbb H$.  We are interested in the relations between the following functional inequalities:

\begin{enumerate}
\item PŁ inequality for the free energy $\mathcal F$:  there exists $\lambda>0$ such that 
\begin{equation}\label{eq:PLF}
\forall \rho \in \mathcal P_2(E)\,,\qquad \mathcal F(\rho)-\inf \mathcal F \leqslant \frac{1}{2\lambda}\mathcal I(\rho)\,.
\end{equation}
\item PŁ inequality over $\mathcal M$ for the coarse-grained free energy $\hF$: there exists $\lambda'>0$ such that
\begin{equation}
\label{eq:PLFhat}
\forall \theta \in \mathbb H,\qquad \hF(m_\theta) - \inf\hF \leqslant \frac{1}{2\lambda'} \left|\C^{-1}\na \hF(m_\theta)\right|^2\,.
\end{equation}
\item PŁ inequality for the renormalised potential $\omega$: there exists $\lambda''>0$ such that 
\begin{equation}
\label{eq:PLomega}
\forall \theta \in \mathbb H,\qquad \omega(\theta) - \inf \omega \leqslant \frac{1}{2\lambda''} |\C\na \omega(\theta)|^2\,.
\end{equation}
\item Ballistic low-temperature LSI for the renormalised measures $(\tilde \nu_N)_{N\geqslant 1}$: for all $N\geqslant 1$, there exists $\tilde \lambda_N> 0$ such that 
\begin{equation}\label{eq:ballisticLSI}
\forall \rho \in \mathcal P_2(\mathbb H),\qquad  \mathcal H(\rho|\tilde \nu_N) \leqslant \frac{1}{2\tilde \lambda_N}\mathcal I_{\C}(\rho|\tilde \nu_N)\,,
\end{equation}
where $\mathcal I_{\C} (\rho|\mu) = \int_{\mathbb H}|\C\na \ln(\rho/\mu)|^2 \rho$, with
\begin{equation}
\label{eq:ballisticlambda_N}
\liminf_{N\rightarrow \infty} \frac{\tilde \lambda_N}{N} >0\,.
\end{equation}
\item Uniform-in-$N$ LSI for the Gibbs measures $(\mu_\infty^N)_{N\geqslant 1}$: for all $N\geqslant 1$, there exists $\lambda_N> 0$ such that 
\begin{equation}\label{eq:unifLSI}
\forall \rho^N \in \mathcal P_2(\R^{dN}),\qquad  \mathcal H(\rho^N|\mu_\infty^N) \leqslant \frac{1}{2\lambda_N}\mathcal I(\rho^N|\mu_\infty^N)\,,
\end{equation}
and $\liminf_{N\rightarrow \infty} \lambda_N >0$.
\end{enumerate}

Our goal is to prove that, under suitable conditions, these inequalities are all equivalent. More precisely, roughly speaking, we will prove that, under additional assumptions,
\begin{equation}
\label{eq:chaineimplications}
\eqref{eq:PLF}\quad \Rightarrow \quad \eqref{eq:PLFhat} \quad \Leftrightarrow \quad  \eqref{eq:PLomega} \quad \Leftrightarrow \quad  \eqref{eq:ballisticLSI}+\eqref{eq:ballisticlambda_N}  \quad \Rightarrow \quad  \eqref{eq:unifLSI}+\eqref{eq:LSI-unifN} \quad \Rightarrow \quad  \eqref{eq:PLF}.
\end{equation}
The additional conditions required will not be the same for each of these implications, and some relations between the various constants $\lambda,\lambda',\lambda'',\tilde \lambda_N$ and $\lambda_N$ will be made precise.

Among all these implications, let us clarify what is already known and what is easy:
\begin{itemize}
\item With Lemma~\ref{lem:omegaFhat} at hand, the implications $\eqref{eq:PLF} \Rightarrow  \eqref{eq:PLFhat}   \Leftrightarrow   \eqref{eq:PLomega} $ are elementary. The corresponding precise statements are gathered in Proposition~\ref{prop:intermediaryimplications}.
\item When $\dim \mathbb H < \infty$, the equivalence $\eqref{eq:PLomega}   \Leftrightarrow    \eqref{eq:ballisticLSI}+\eqref{eq:ballisticlambda_N}$ is precisely the result of Chewi and Stromme in  \cite{chewi2024ballistic} (recalled in Theorem~\ref{thm:ChewiStromme} below), in fairly general settings. We also mention \cite[Lemma 6]{Monmarchemetastable} which establishes a ballistic LSI for $\tilde \nu_N$ under a slightly different condition that the PŁ inequality~\eqref{eq:PLomega} (still in the case $\dim \mathbb H < \infty$, with dimension dependencies in the constants).
\item The implication $\eqref{eq:ballisticLSI}+\eqref{eq:ballisticlambda_N}    \Rightarrow    \eqref{eq:unifLSI}+\eqref{eq:LSI-unifN}$ is what is obtained with the method of~\cite{Dagallier}. Since this work is restricted to pairwise interactions energies of the form~\eqref{eq:EVW}, for the sake of completeness we adapt the proof of~\cite{Dagallier} to the case~\eqref{eq:energy-quadra}, leading to Theorem~\ref{thm:PL->unifLSI} (besides, we clarify some points related to the case $\dim \mathbb H=\infty$).
\item As discussed in the introduction, the implication $\eqref{eq:unifLSI}+\eqref{eq:LSI-unifN}\Rightarrow \eqref{eq:PLF}  $ has been established for pairwise interaction in \cite{Pavliotis,GuillinWuZhang}. We state a result for the general form~\eqref{eq:energy-quadra} in Theorem~\ref{thm:unifLSI->PL}.
\end{itemize}

The precise statements are provided in the next section.

\subsection{Intermediary results}\label{sec:intermediaryresults}

\subsubsection{Direct implications}

Among~\eqref{eq:chaineimplications}, the easy implications are the following:

\begin{proposition}\label{prop:intermediaryimplications}
Under Assumptions~\ref{assum:finitedim} and \ref{assum:baseLemme1}, the following holds:
\begin{enumerate}
\item If $\varphi$ is $\mathcal C^1$, Lipschitz continuous and not constant,  the PŁ inequality~\eqref{eq:PLF} for $\mathcal F$ implies the PŁ inequality~\eqref{eq:PLFhat} for $\hF$ with $\lambda' = \lambda/\|\na \varphi\|_\infty^2$.
\item   The PŁ inequality~\eqref{eq:PLFhat} for $\hF$ implies the PŁ inequality~\eqref{eq:PLomega} for $\omega$ with $\frac1{\lambda''} = 1+\frac1{\lambda'}$.
\item   If $\omega$ satisfies the PŁ inequality~\eqref{eq:PLomega} and is not constant, then necessarily $\lambda''<1$ and the PŁ inequality~\eqref{eq:PLFhat} holds with $\frac{1}{\lambda'} = \frac{1}{\lambda''}-1$.
\end{enumerate}
\end{proposition}

The proof is postponed to Section~\ref{sec:proof-intermediary}. The restriction to the case $\dim \mathbb H < \infty$ is for commodity: since the constants are fully explicit and dimension-free, following Remark~\ref{rem:diminfini} we get the same result in the case $\dim\mathbb = \infty$ for a suitable sense of the PŁ inequalities~\eqref{eq:PLFhat} and~\eqref{eq:PLomega} (notice that to get dimension-free constants in~\eqref{prop:intermediaryimplications} it was crucial to involve $\C$ in the definition of the inequalities~\eqref{eq:PLFhat} and~\eqref{eq:PLomega}).

\subsubsection{Finite-dimensional PL and minimiser uniqueness imply ballistic LSI}

For $\eqref{eq:PLomega}   \Leftrightarrow    \eqref{eq:ballisticLSI}+\eqref{eq:ballisticlambda_N}$, we simply recall the result of Chewi and Stromme: 

\begin{theorem}[from {\cite[Theorem 1]{chewi2024ballistic}}]\label{thm:ChewiStromme}
Under Assumption~\ref{assum:finitedim}, suppose furthermore that $\omega\in \mathcal C^2(\mathbb H)$ with $\Delta \omega \leqslant L (1+|\na \omega|^2)$ for some $L>0$, and that $\omega$ admits a unique global minimizer and satisfies the PŁ inequality~\eqref{eq:PLomega}. Then, for all $N\geqslant 1$, $\tilde \nu_N$ satisfies the LSI~\eqref{eq:ballisticLSI} with a constant $\tilde \nu_N$ such that
\begin{equation}
\label{loc:zfgldazmlmf}
 \frac{\tilde \lambda_N}{N} \underset{N\rightarrow \infty}\longrightarrow \lambda''\,. 
\end{equation}
\end{theorem}

Some comments:
\begin{itemize}
\item Notice that \cite[Theorem 1]{chewi2024ballistic} is a priori restricted to the case $\C=I$, but is suffices to apply it to the preconditioned potential $\theta \mapsto \omega(\C \theta)$ and then use that $\tilde \nu_N$ is the image of the corresponding preconditioned Gibbs measure by $\theta \mapsto \C \theta$ to get Theorem~\ref{thm:ChewiStromme} in this form. 
\item The bound on $\Delta \omega$ required in the theorem can be obtained as in Lemma~\ref{lem:regularitycheck} below.
\item The condition that $\omega$ admits a unique minimizer is necessary, cf. the example in Section~\ref{sec:contreexemple}. Otherwise, $\tilde \lambda_N$ is of order $1$ as $N\rightarrow \infty$, with $\tilde \lambda_N^{-1}$ corresponding to the time-scale of a diffusive behavior in the flat set of the global minimizers.
\item As already mentioned in Remark~\ref{rem:diminfini}, contrary to Proposition~\ref{prop:intermediaryimplications}, here the restriction to the case $\dim \mathbb H <\infty$ is important since the error terms in the convergence~\eqref{loc:zfgldazmlmf} depends on the dimension (they are of order $\dim \mathbb H$). In particular if we follow the proof of \cite[Theorem 1]{chewi2024ballistic} to get an explicit rank $N_0$ such that $\tilde \lambda_N \geqslant \lambda''N/2$ for all $N\geqslant N_0$ then this $N_0$ depends on $\dim\mathbb H$.
\end{itemize}

\subsubsection{Uniform LSI implies non-linear PŁ}

 The fact that a uniform LSI~\eqref{eq:unifLSI}+\eqref{eq:LSI-unifN} implies a non-linear PŁ inequality~\eqref{eq:PLF} (with $\lambda\geqslant \bar \lambda := \limsup_{N\rightarrow \infty}\lambda_N$) has been established under various settings, for instance \cite[Theorem 3.3]{Pavliotis}, \cite[Theorem 10]{GuillinWuZhang} or \cite[Theorem 11]{Monmarchemetastable}. The proof is based on the fact that, under suitable conditions and for  $\rho$ in a suitable class of  probability densities over $E$,
 \begin{equation}
 \label{eq:grandeDev}
 \frac1N   \mathcal H(\rho^{\otimes N}|\mu_\infty^N) \underset{N\rightarrow\infty}\longrightarrow \mathcal F(\rho) - \inf \mathcal F\,, \qquad \frac1N \mathcal I(\rho^{\otimes N}|\mu_\infty^N) \underset{N\rightarrow\infty }\longrightarrow \mathcal I(\rho) \,.
 \end{equation}
 In particular, applying~\eqref{eq:unifLSI} with $\rho^N=\rho^{\otimes N}$, dividing by $N$ and passing to the limit gives~\eqref{eq:PLF}. Moreover, as in \cite[Theorem 3.6]{Pavliotis}, the uniform LSI~\eqref{eq:unifLSI} implies that $\mathcal F$ admits a unique critical point, which is its global minimiser, and which is the limit of the one-particle marginal distribution under $\mu_\infty^N$. Indeed, the PŁ inequality~\eqref{eq:PLF} implies that all critical points are global minimisers and, using the Talagrand inequality implied by the LSI~\cite{OttoVillani}, if $\rho_*$ is such a minimisers, denoting by $\mu_\infty^{N,1}(x_1) = \int_{E^{N-1}} \mu_\infty^N(\bx)\dd x_2\dots\dd x_N$ the one-particle marginal of the Gibbs measure, by exchangeability of the latter,
\begin{align}
\mathcal W_2^2(\rho_*,\mu_\infty^{N,1}) & \leqslant \frac1N \mathcal W_2^2(\rho_*^{\otimes N},\mu_\infty^N) \nonumber \\
& \leqslant \frac{2}{N \lambda_N}   \mathcal H(\rho_*^{\otimes N}|\mu_\infty^N)   \nonumber\\
& \leqslant \frac{2}{ \bar \lambda-o(1)} \po \mathcal F(\rho_*) - \inf \mathcal F +o(1) \pf \underset{N\rightarrow \infty}\longrightarrow 0.\label{loc:argumentuniqueness}
\end{align}
In particular, the minimiser is unique.

Since previous works mostly focused on models with pairwise interaction, for the sake of completeness we now give general conditions under which the discussion above can be made rigorous.

\begin{assumption}
\label{assum:unifLSI->PL}
The energy $\mathcal E$ is semi-lower continuous (with respect to weak convergence). For any $\mu \in \mathcal P_2(\R^d)$, $D\mathcal E(\mu,\cdot)$ exists, is  continuous, and such that
\begin{equation}
\label{eq:CVDE}
\int_{E^N} |D\mathcal E(\pi_{\bx},x_1) - D\mathcal E(\mu,x_1)|^2 \mu^{\otimes N}(\dd \bx) \underset{N\rightarrow\infty}\longrightarrow 0\,.
\end{equation}
Moreover, for all $\nu\in\mathcal P_2(\R^d)$, the measure $\mu \propto \exp(-\frac{\delta \mathcal E}{\delta m}(\nu,\cdot))$ is such that
\begin{equation}
\label{eq:CVE}
\int_{E^N} \mathcal E(\pi_{\bx}) \mu^{\otimes N}(\dd \bx) \underset{N\rightarrow\infty}\longrightarrow \mathcal E(\mu)\,.
\end{equation}
\end{assumption}

\begin{example}
For 
\[\mathcal E(\mu) = \int_{E} V \dd \mu + \frac12 \int_{E^2} W(x,y) \mu(\dd x)\mu(\dd y)\,,\]
assuming that $W$ is symmetric ($W(x,y)=W(y,x)$) and $\mathcal C^1$ and, if $E=\R^d$, that there exists $C>0$ such that $|\na_x W(x,y)| \leqslant C(|x|+|y|+1)  $ for all $x,y\in\R^d$,   then $D\mathcal E(\mu,x) = \int_{\R^d} \na_x W(x,y) \mu(\dd y)$ and
\begin{multline*}
\int_{E^N} |D\mathcal E(\pi_{\bx},x_1) - D\mathcal E(\mu,x_1)|^2 \mu^{\otimes N}(\dd \bx) 
= \frac1{N^2} \int_{E} \left|\na_x W(x_1,x_1) - D\mathcal E(\mu,x_1)\right|^2 \mu(\dd x_1)\\ 
+ \frac{N-1}{N^2} \int_{E^2} \left|\na_x W(x_1,x_2) - D\mathcal E(\mu,x_1)\right|^2\mu(\dd x_1)\mu(\dd x_2) \underset{N\rightarrow\infty}\longrightarrow 0\,,
\end{multline*}
for all $\mu \in \mathcal P_2(E)$. Similarly,
\begin{align*}
\int_{E^N} \mathcal E(\pi_{\bx}) \mu^{\otimes N}(\dd \bx) &  = \int_{E} V \dd \mu + \frac{1}{2N} \int_{E} W(x,x)\mu(\dd x ) +  \po \frac{1}{2}-\frac{1}{2N}\pf \int_{E^2} W(x,y)\mu(\dd x)\mu(\dd y) 
\end{align*}
converges to $\mathcal E(\mu)$ for all $\mu \in \mathcal P_2(E)$ as $N\rightarrow\infty$, a fortiori for all  $\mu \propto \exp(-\frac{\delta \mathcal E}{\delta m}(\nu,\cdot))$ with $\nu\in\mathcal P_2(E)$ provided $\mu \in \mathcal P_2(E)$, which is easily checked when $E=\R^d$ under suitable growth conditions on $V$ with respect to $W$ since $\frac{\delta \mathcal E}{\delta m}(\nu,x) = V(x) + \int_{\R^d} W(x,y)\nu(\dd y)$.
\end{example}

\begin{theorem}\label{thm:unifLSI->PL}
Under Assumption~\ref{assum:unifLSI->PL}, if for each $N\geqslant 1$ the Gibbs measure $\mu_\infty^N$ satisfies a LSI~\eqref{eq:unifLSI} and $\bar \lambda := \limsup_{N\rightarrow \infty}\lambda_N >0$, then $\mathcal F$ admits a unique critical point $\mu_*\in\mathcal P_2(E)$, which is its global minimiser and is the limit in $\mathcal P_2(E)$ of $\mu_\infty^{N,1}$. Moreover, the PŁ inequality~\eqref{eq:PLF} holds with $\lambda \geqslant \bar \lambda$.
\end{theorem}

The proof is postponed to Section~\ref{sec:proof-nonlinearPL}.

\subsubsection{Ballistic coarse-grained LSI implies uniform LSI}\label{sec:PL->LSI}

The remaining implication of~\eqref{eq:chaineimplications} is \eqref{eq:ballisticLSI}$+$\eqref{eq:ballisticlambda_N} $\Rightarrow$ \eqref{eq:unifLSI}$+$\eqref{eq:LSI-unifN}. The proof relies on the two-scale approach for proving log-Sobolev inequalities, as developed in~\cite{bauerschmidt2019very}. As a matter of fact, the implication \eqref{eq:ballisticLSI}$+$\eqref{eq:ballisticlambda_N} $\Rightarrow$ \eqref{eq:unifLSI}$+$\eqref{eq:LSI-unifN} has been proven along the proof of \cite[Theorem 1.3]{Dagallier} (which actually essentially states that~\eqref{eq:PLFhat} implies~\eqref{eq:unifLSI}$+$\eqref{eq:LSI-unifN}), in the case of a quadratic pairwise interaction. For more general pairwise interactions, in the same work, the authors also prove in \cite[Theorem 4.2]{Dagallier} a uniform LSI for $\mu_\infty^N$ when $\omega$ is strongly convex (which is stronger than the PŁ inequality), see Section~\ref{sec:convex} on this topic. As discussed in \cite[Section 1.4]{Dagallier}, the proof still works if the convexity condition is replaced by a PŁ inequality, at least when $\dim \mathbb H<\infty$ since \cite{chewi2024ballistic} has to be applied to get the low-temperature LSI~\eqref{eq:ballisticLSI}$+$\eqref{eq:ballisticlambda_N}, instead of the dimension-free Bakry-Emery criterion in the strongly convex case. Since in our presentation we have splitted the various implications~\eqref{eq:chaineimplications}, in this section we just consider that the LSI~\eqref{eq:ballisticLSI}$+$\eqref{eq:ballisticlambda_N}  is given and we follow the end of the proof of~\cite[Theorem 4.2]{Dagallier} to conclude. Since, as mentioned in Section~\ref{sec:parametrized}, we are interested in non-pairwise interactions appearing in numerical applications and the modified energy method of~\cite{Monmarchemetastable}, we consider energies of the form~\eqref{eq:energy-quadra}, slightly extending the settings of~\cite{Dagallier}. Moreover we clarify the dependency with respect to $\mathrm{Tr}(\C^2)$ (compare Lemma~\ref{lem:gammaX} to \cite[Equation (4.51)]{Dagallier}).

\begin{assumption}
\label{assum:PL->unifLSI}
    The energy $\mathcal E$ is of the form~\eqref{eq:energy-quadra} such that:
    \begin{itemize}
        \item The external potential $V\in\mathcal C^2(E,\R)$. Moreover, if $E=\R^d$, it can be decomposed as $V=V_c+V_b+V_L$ where $V_b$ is bounded, $V_L$ is Lipschitz continuous and  $V_c$ is  strongly convex.
        \item The energy $\mathcal E_c$ is flat-convex and bounded, the flat derivatives $\frac{\delta \mathcal E_c}{\delta m}$ and $\frac{\delta^2 \mathcal E_c}{\delta m^2}$ exist, are jointly continuous, bounded, and $\mathcal C^2$ in the space variables. Moreover,  $\na_x\frac{\delta^2 \mathcal E_c}{\delta m^2}$ and $D^2\mathcal E_c$ are bounded.
        \item The parameter map $\varphi$ is Lipschitz continuous and can be written as $\varphi=\C\psi$ with $\psi \in \mathcal C^1(E,\mathbb H)$ and $\C$ a self-adjoint operator such that $\C^2$ is trace-class. We can decompose $\rp=(\rp_\ell,\rp_b) \in \mathbb H_b\times\mathbb H_\ell$ where $\rp_\ell$ is linear and $\rp_b$ is bounded.
    \end{itemize}
\end{assumption}

Under Assumption~\ref{assum:PL->unifLSI}, let $M_b,L_\ell,\kappa_c,M_c,L_\varphi,M_\psi,T_{\C}>0$ be such that  $V_b$ is $M_b$-bounded, $V_L$ is $L_\ell$-Lipschitz continuous, $V_c$ is  $\kappa_c$-strongly convex,  $\|\mathcal E_c\|_\infty + \|\na_x\frac{\delta^2 \mathcal E_c}{\delta m^2}\|_\infty +\| D^2\mathcal E_c\|_\infty \leqslant M_c$, $\varphi$ is $L_\varphi$-Lipschitz, $\|\rp_b\|_\infty \leqslant M_\psi$ and $\mathrm{Tr}(\C^2) \leqslant T_{\C}$. Gather all these parameters in
\begin{equation}
\label{eq:param}
\mathfrak{P} =\{ M_b,L_\ell,\kappa_c,M_c,L_\varphi,M_\psi,T_{\C}\} \,.
\end{equation}
As established in Lemma~\ref{lem:omegaNtoomega}, under Assumption~\ref{assum:PL->unifLSI}, the probability measure $\tilde \nu_N$ in~\eqref{eq:deftildenuN} is well-defined.

\begin{theorem}\label{thm:PL->unifLSI}
Under Assumption~\ref{assum:PL->unifLSI}, there exists $K>0$ which depends only on the parameters $\mathfrak{P}$ in~\eqref{eq:param} such that, for any $N\geqslant 1$, if the normalised measure $\tilde \nu_N$ satisfies a LSI~\eqref{eq:ballisticLSI} then the Gibbs measure $\mu_\infty^N$ satisfies a LSI~\eqref{eq:unifLSI} with  $\lambda_N^{-1} \leqslant K(1+N\tilde \lambda_N^{-1})$. In particular, if the ballistic scaling~\eqref{eq:ballisticlambda_N} holds then the LSI for $\mu_\infty^N$ is uniform in $N$, i.e.~\eqref{eq:LSI-unifN} holds. 
\end{theorem}

The proof is the topic of Section~\ref{sec:proofLSIN}. As it is constructive, in principle it provides an explicit lower bound on $\underline{\lambda}=\liminf \lambda_N$. However it doesn't give a nice expression, and clearly combining Theorems~\ref{thm:PL->unifLSI} and \ref{thm:ChewiStromme} with Proposition~\ref{prop:intermediaryimplications} doesn't give that $\underline{\lambda} \geqslant \lambda$, i.e. it doesn't prove \cite[Conjecture 1]{Pavliotis}. 

\subsection{Final general result}\label{sec:finalresult}

Concatenating all the previous intermediary results in order to answer to the question raised in Section~\ref{sec:motivation} leads to the following.

\begin{theorem}\label{thm:final}
Under Assumptions~\ref{ass:basic}, \ref{assum:finitedim}, \ref{assum:baseLemme1}, \ref{assum:unifLSI->PL} and~\ref{assum:PL->unifLSI}, the following are equivalent:
\begin{enumerate}
\item \emph{(uniform LSI)} For all $N\geqslant 1$, the $N$-particle log-Sobolev constant~\eqref{eq:deflambdaNLSI} is positive. Moreover, $\liminf_{N\rightarrow \infty}\lambda_N >0$.
\item \emph{(non-linear PŁ and uniqueness)} The PŁ constant~\eqref{eq:deflambdaPL} is positive, and $\mathcal F$ admits a unique global minimiser.
\end{enumerate}
\end{theorem}

\new{
For clarity, let us provide simple settings where all the conditions required in Theorem~\ref{thm:final} are met. Since we will consider only the case $\dim \mathbb H<\infty$ and we won't try to make any constant explicit, we can focus on the case $\C=I$, $\varphi=\psi$.
}

\new{
\begin{proposition}\label{prop:assumfinal}
    Assume that the energy $\mathcal E$ is of the form
    \begin{equation}
    \label{eq:sfmsdfgdfg}
    \mathcal E\po \mu \pf = \mu(V) + R\po \mu( \varphi_b )\pf -  \left|\int_{E} Ax \mu(\dd x)\right|^2 \,,
    \end{equation}
    with the following conditions
    \begin{itemize}
        \item The external potential $V\in\mathcal C^2(E,\R)$. Moreover, if $E=\R^d$, it can be decomposed as $V=V_c+V_b+V_L$ where $V_b$ is bounded, $V_L$ is Lipschitz continuous and  $V_c$ is  strongly convex.
              \item The parameter map $\varphi_b :E\rightarrow \mathbb H$  is Lipschitz continuous and  bounded.  If $E=\T^d$, $A=0$, and if $E=\R^d$, $A$ is a $r\times d$ matrix for some $r\in\N$.
              \item The dimension of $\mathbb H$ is finite.
              \item The function $R\in\mathcal C^2(\mathbb H,\R)$ is semi-convex (i.e. $\theta \mapsto R(\theta) + \kappa |\theta|^2$ is convex for some $\kappa>0$) with a bounded second order derivative.
        \end{itemize}
        Then Assumptions~\ref{ass:basic}, \ref{assum:finitedim}, \ref{assum:baseLemme1}, \ref{assum:unifLSI->PL} and~\ref{assum:PL->unifLSI} holds, and thus Theorem~\ref{thm:final} applies.
\end{proposition}
}

\new{To prove Proposition~\ref{prop:assumfinal}, the only non-straightforward verifications are, first, to check Assumption~\ref{assum:baseLemme1}: it follows from Lemmas~\ref{lem:uniqueminimumtheta} and~\ref{lem:regularitycheck}; and, second, the bound on $\Delta \omega$ required in Theorem~\ref{thm:ChewiStromme}: Lemma~\ref{lem:regularitycheck} shows that $\na^2 \omega$ is bounded. }

\begin{example}
Consider an external potential $V\in\mathcal C^2(\R)$ which is the sum of a strongly convex and a bounded functions, and a pairwise interaction of the form
\[W(x,y) = \frac12 |A(x-y)|^2 + \sum_{k=1}^m \co a_k \cos\po \nu_k(x-y )\pf + b_k \sin\po \nu_{k}(x-y )\pf\cf \]
for some  $r\times d$ matrix $A$ and some coefficients $(a_k,b_k,\nu_k)_{k\in\cco 1,m\ccf} \in \R^{3m}$, $m\in\N$. Assume that
\[\forall x\in\R^d,\qquad
 V(x) \geqslant \varepsilon|x|^2 +  |Ax|^2 -C\]
for some $C,\varepsilon>0$. Expanding the sines and cosines we can write $\mathcal E(\mu)=\mu(V)+ \frac12 \mu^{\otimes 2}(W)$ in the form~\eqref{eq:sfmsdfgdfg} with
\[\varphi(x) = \po A x, (\cos(\nu_k x),\sin(\nu_k x))_{k\in\cco 1,m\ccf}\pf \in \mathbb H = \R^{r+2m}\,,\]
and a quadratic function $R$.  With the motivation provided in~\cite{Monmarchemetastable} we can consider more generally an energy of the form
\[\mathcal E_m(\mu) = \mu(V) + \frac12 \mu^{\otimes 2} (W) + h \po \mu(\varphi) \pf  \]
where $h$ is a $\mathcal C^2$ bounded function with $\na h$ and $\na^2 h$ bounded. Then all the conditions required in Proposition~\ref{prop:assumfinal} holds.
\end{example}

\subsection{Convexity conditions}\label{sec:convex}

In the case where possibly $\dim \mathbb H = \infty$, we cannot apply~\cite{chewi2024ballistic} to get a ballistic renormalised LSI~\eqref{eq:ballisticLSI}, but we can still apply the Bakry-Emery criterion if $\omega$ is strongly convex, as has been considered in \cite{Dagallier}. In \cite[Lemma 4.3]{Dagallier}, this strong convexity is shown to be equivalent to the one of $\hF$. In the spirit of Proposition~\ref{prop:intermediaryimplications}, we  now show that this can be deduced from a suitable strong convexity condition on the original free energy $\mathcal F$ itself. Moreover we show a possibly local version of this property, which may be of interest in order to generalize the results of~\cite{chewi2024ballistic} to infinite dimension (since the latter rely on the local strong convexity around the local minimiser, which in finite dimension is implied by the PŁ inequality and the minimiser uniqueness -- this local convexity is also used in \cite[Lemma 6]{Monmarchemetastable} for a similar result), and to apply the modified energy method of~\cite{Monmarchemetastable} in more general settings.

\begin{lemma}\label{lem:convexite}
Assume that for all $\theta \in \mathbb H$, $\mathcal F_\theta$
admits over $\mathcal P_2(E)$ a unique global minimiser $\bar\mu_\theta \in \mathcal P_2(E)$, and that there exists $\kappa>0$ and a convex subset $\mathcal B$ of $\mathbb H$  such that
\begin{equation}
\label{eq:convexiteF}
\forall t\in[0,1],\qquad t \mathcal F(\mu) + (1-t) \mathcal F(\rho) \geqslant  \mathcal F( t \mu +(1-t)\rho) +  \frac{\kappa}{2}t(1-t) |\mu(\varphi)-\rho(\varphi)|^2
\end{equation}
for all $\mu,\rho \in \mathcal P_2(E)$ such that $\mu(\psi),\rho(\psi) \in \mathcal B$. Then, for all $\theta ,\zeta\in \mathbb H$ such that $m_\theta:=\hat \mu_\theta(\psi)$ and $m_\zeta$ are in $\mathcal B$,
\begin{equation}
\label{eq:convexiteFhat}
\forall t\in[0,1],\qquad  t \hF(m_\theta) + (1-t) \hF(m_\zeta) \geqslant \hF( t m_\theta +(1-t)m_\zeta) +  \frac{\kappa}{2}t(1-t) |\C (m_\theta-m_\zeta)|^2\,.
\end{equation}
In turns, this inequality implies that for all $\theta,\zeta\in\mathbb H$ with $m_\theta,m_\zeta\in\mathcal B$,
\begin{equation}
\label{eq:convexiteomega}
\forall t\in[0,1],\qquad t\omega(\theta) + (1-t) \omega(\zeta) \geqslant \omega(t\theta + (1-t)\zeta)  +   \frac{\kappa }{2(1+\kappa)}t(1-t) |\C^{-1} ( \theta-\zeta)|^2\,.
\end{equation}
\end{lemma}

In the case where $\mathcal B= \mathbb H$,  by the Bakry-Emery criterion, the convexity condition~\eqref{eq:convexiteomega} implies the ballistic LSI~\eqref{eq:ballisticLSI} with $\tilde \lambda_N \geqslant N\kappa/(1+\kappa)$, and Theorem~\ref{thm:PL->unifLSI} can be applied to deduce a uniform-in-$N$ LSI~\eqref{eq:LSI-unifN}. As a conclusion, we get the following, which extends   \cite[Theorem 4.2]{Dagallier} to non-pairwise interactions and state the convexity condition in terms of $\mathcal F$ instead of $\hF$:

\begin{theorem}
\label{thm:generalconvexity}
Under Assumptions~\ref{assum:baseLemme1}, \ref{assum:unifLSI->PL} and~\ref{assum:PL->unifLSI}, suppose furthemore that there exists $\kappa>0$ such that~\eqref{eq:convexiteF} holds for all $\mu,\rho \in \mathcal P_2(E)$. Then for all $N\geqslant 1$ the LSI~\eqref{eq:unifLSI} holds, with $\liminf_{N\rightarrow \infty} \lambda_N >0$.
\end{theorem}

\section{Counter-example without minimiser uniqueness}\label{sec:contreexemple}

Since a PŁ inequality doesn't imply uniqueness of the global minimiser, in view of Theorem~\ref{thm:final} it is easy to construct counter-examples to the implication
\[\lambda>0\qquad \Rightarrow \qquad \liminf_{N\rightarrow \infty} \lambda_N>0\]
when the uniqueness condition is not satisfied. Let us provide explicitly such a counter-example. Given $R\in \mathcal C^2(\R,\R)$, consider for $\rho\in\mathcal P_2(\R)$ the energy
\begin{equation}
\label{eq:Etoymodel}
\mathcal E\po \rho\pf = \frac12 \int_{\R} |x|^2 \rho(\dd x) + R\po \rho(\varphi)\pf  + \frac{1}{2}\ln (2\pi)
\end{equation}
with $\varphi(x) = x$. This toy model is extensively studied in~\cite{Mtoymodel}. The corresponding free energy $\mathcal F=\mathcal E + \mathcal H$ can be written
\[\mathcal F(\rho) = f\po\mu(\varphi)\pf + \mathcal H(\mu|\gamma_{\mu(\varphi)}) \,, \] 
with $f(\theta)= R(\theta) + \frac12|\theta|^2 $ and $\gamma_\theta = \mathcal N(\theta,1)$. In particular, $\mathcal F(\rho) > \mathcal F(\gamma_{\rho(\varphi)})$ if $\rho \neq \gamma_{\rho(\varphi)}$, and since $\gamma_{\rho(\varphi)}(\varphi) = \rho(\varphi)$, we get that, for all $\theta\in\R$, 
\[\hF(\theta) = \mathcal F(\gamma_{\theta}) = f(\theta) \,. \]
Moreover, the set of minimizers of $\mathcal F$ is exactly $\{\gamma_\theta\, : \, f(\theta)=\inf f\}$, and $\inf \mathcal F=\inf f$. By Lemma~\ref{lem:omegaFhat}, we also get that the set of minimizers of $\omega$ coincide with the one of $f$ (in fact $\omega$ is a Moreau  envelope of $f$ and thus this is a general fact). If $f$ admits an open interval of minimizers, then it is easy to see that the log-Sobolev constant of $\tilde \nu_N \propto e^{-N \omega}$ cannot be larger than the one of the uniform measure over this interval. Since the latter is independent from $N$, the ballistic regime~\eqref{eq:ballisticlambda_N} doesn't hold. In other words, in the chain of implications~\eqref{eq:chaineimplications}, we see that the assumption that the minimiser if unique is necessary in the step~\eqref{eq:PLomega}$\Rightarrow$ \eqref{eq:ballisticLSI}$+$\eqref{eq:ballisticlambda_N}.

To design easily an explicit counter-example to~\eqref{eq:equivalence}, we rely on the following.

\begin{lemma}\label{lem:toymodellemma}
For the model~\eqref{eq:Etoymodel}, if $f$ satisfies a PŁ inequality 
\[\forall \theta\in \R,\qquad f(\theta) - \inf f \leqslant \frac{1}{2\lambda'}|\na f(\theta)|^2 \]
for some $\lambda'>0$, then the PŁ constant~\eqref{eq:deflambdaPL} of $\mathcal F$ satisfies $\lambda \geqslant \min(1,\lambda')$.
\end{lemma}

This is one of the statements in \cite[Proposition 4]{Mtoymodel}, to which we refer for the proof. Notice that, in this proposition, it is assumed that $f$ has a unique global minimizer, but this is not used in the proof of the implication $(i)\Rightarrow(ii)$ in \cite[Proposition 4]{Mtoymodel}, which is the one that gives Lemma~\ref{lem:toymodellemma}.

With this lemma, it is easy to construct a free energy $\mathcal F$ with several minimisers and satisfying a PŁ inequality, for instance we take the energy~\eqref{eq:Etoymodel} with
\begin{equation}
\label{eq:specifictoy}
R(\theta) = f(\theta) - \frac{1}{2} \theta^2 \qquad f(\theta) =  (|\theta|-1)_+^2\,. 
\end{equation}
Notice that $R$ is $\mathcal C^1$, has a lower-bounded curvature and is convex outside $[-1,1]$. For $\theta>1$,
\[|f'(\theta)|^2 = 4|\theta-1|^2 \geqslant 4 \po f(\theta)-\inf f\pf\,, \]
and similarly for $\theta<-1$. As a consequence, $\mathcal F$ satisfies a PŁ inequality with $\lambda=1$.

Now, turning our attention to the Gibbs measure $\mu_\infty^N$, we notice the following;

\begin{lemma}\label{lem:marginaltoymodel}
For the model~\eqref{eq:Etoymodel}, if $\bX\sim \mu_\infty^N$, then $\bar X := \frac1N\sum_{i=1}^N X_i$ is distributed according to $\tilde \nu_N \propto \exp(-Nf)$.
\end{lemma}
\begin{proof}
Writing $\bar x= \frac1N\sum_{i=1}^N x_i = N^{-1/2} \bx \cdot u $ with the unit vector $u=N^{-1/2}(1,\dots,1)$, for any bounded test function $h$, an orthogonal change of variables gives
\begin{align*}
\int_{\R^N} h(\bar x) e^{-N R(\bar x)} e^{-\frac12|\bx|^2} \dd \bx &= \int_{\R^N} h(x_1/\sqrt N) e^{-N R(x_1/\sqrt N)} e^{-\frac12|\bx|^2} \dd \bx \\
& \propto \int_{\R} h(y) e^{-N R(y)} e^{-\frac N2|y|^2} \dd y\,,
\end{align*}
which concludes the proof since $f(\theta)= R(\theta) + \frac12|\theta|^2 $.
\end{proof}

Since a LSI implies a Poincaré inequality with the same constant, namely
\[\lambda_N \leqslant   \inf\left\{ \frac{\mu_\infty^N(|\na g|^2)}{2\mu_\infty^N(g^2)}\,:\, g\in L^2(\mu_\infty^N),\ g \neq \mu_\infty^N(g) = 0  \right\}\,,\]
upper-bounds on $\lambda_N$ are easily obtained with Lemma~\ref{lem:marginaltoymodel} by considering test functions of the form $g(\bx) = h(\bar x)$. For instance in the specific symmetric case~\eqref{eq:specifictoy} we can simply take $g(\bx) = \bar x$, for which $|\na g(\bx)|^2 = 1/N$ for all $\bx$ and
\[\mu_\infty^N(g^2) = \mathbb E \po |\bar X|^2\pf = \frac{\int_{\R}y^2 e^{-Nf(y)}\dd y }{\int_{\R}  e^{-Nf(y)}\dd y } \underset{N\rightarrow \infty}\longrightarrow \frac12 \int_{-1}^1 y^2 \dd y = \frac13\,. \]
As a conclusion, we have obtained the following:

\begin{proposition}\label{prop:counterexample}
For the model~\eqref{eq:Etoymodel} with $R$ given by~\eqref{eq:specifictoy}, the PL and $N$-particle log-Sobolev constants~\eqref{eq:deflambdaPL} and~\eqref{eq:deflambdaNLSI} are such that
\[\lambda \geqslant 1\,,\qquad \lambda_N \leqslant \frac{3}{2N}\po 1+ \underset{N\rightarrow \infty}o(1)\pf\,.  \]
In particular, the equivalence~\eqref{eq:equivalence} does not hold.
\end{proposition}

\begin{remark}
It could be tempting to think that a perturbation of this counter-example could also give a counter-example with a unique minimizer of the conjecture from~\cite{Pavliotis} according to which $\lambda = \lim \lambda_N$. For instance, starting from~\eqref{eq:specifictoy}, we could consider $R_\varepsilon(\theta)=R(\theta) + \varepsilon\theta^2$: the minimiser is now unique, but $R_\varepsilon\rightarrow R$ as $\varepsilon\rightarrow 0$. Because of Proposition~\ref{prop:counterexample}, we can think that the liminf of $\lambda_N$ as $N\rightarrow \infty$ can be made arbitrarily small when taking $\varepsilon$ small enough (which is indeed the case), and this would give a counter-example if the PŁ constant of the free energy $\mathcal F_\varepsilon$ associated to $R_\varepsilon$ were bounded below uniformly in $\varepsilon \in(0,1]$, which would be the case if this PŁ constant were converging to the one of $\mathcal F$ associated to $R$. Unfortunately this is not the case: for $\theta\in[-1,1]$, $f_\varepsilon(\theta):=f(\theta) + \varepsilon|\theta |^2 = \varepsilon|\theta|^2$ and $f_\theta'(\theta) = 2\varepsilon\theta$, from which the PŁ constant of $f_\varepsilon$ is smaller than (equal to in fact) $2\varepsilon$.  From \cite[Proposition 4]{Mtoymodel}, the PŁ constant of $\mathcal F_\varepsilon$ is $\min(2\varepsilon,1)$.

In fact there is no hope to design a counter-example to $\lambda=\lim \lambda_N$ with a free energy of the form~\eqref{eq:Etoymodel}, because the conjecture from~\cite{Pavliotis} is true in that case (as soon as the minimiser is unique), as shown in \cite[Proposition 4]{Mtoymodel}.
\end{remark}
 
 \section{Degenerate Łojasiewicz inequalities}\label{sec:degenerate}
 
 We consider the same settings and notations as in Section~\ref{sec:results}.

 For a non-decreasing $\Phi: \R_+ \rightarrow \R_+$, we say that a lower-bounded function $f \in \mathcal C^1(\R^d)$ satisfies a $\Phi$-Łojasiewicz inequality (abbreviated as $\Phi$-ŁI) if
 \[\forall x\in\R^d,\qquad f(x) - \inf f \leqslant \Phi(|\na f(x)|^2)\,.\]
 Necessarily, $\Phi(r)>0$ for $r>0$, since a minimizer of $f$ is   a critical point.  We say that the inequality is tight if $\Phi$ is continuous at $0$ with $\Phi(0)=0$, defective otherwise.
 
 We consider the same settings as in Section~\ref{sec:coarsegrainedIneq}, with some $\mathbb H$, $\C$ and $\varphi=\C\psi$.  Under Assumption~\ref{assum:baseLemme1}, we can consider the same group of inequalities:
 \begin{enumerate}
\item Free energy:   there exists a non-decreasing $\Phi:\R_+\rightarrow\R_+$  such that 
\begin{equation}\label{eq:PLF-dege}
\forall \rho \in \mathcal P_2(E)\text{ with }\mathcal F(\rho)<\infty,\qquad \mathcal F(\rho)-\inf \mathcal F \leqslant \Phi\po \mathcal I(\rho)\pf \,.
\end{equation}
\item Coarse-grained free energy: there exists a non-decreasing $\widehat{\Phi}:\R_+\rightarrow\R_+$   such that
\begin{equation}
\label{eq:PLFhat-dege}
\forall \theta \in \mathbb H,\qquad \hF(m_\theta) - \inf\hF \leqslant \widehat{\Phi}\po  \left|\C^{-1}\na \hF(m_\theta)\right|^2\pf \,.
\end{equation}
\item Renormalised potential: there exists a non-decreasing $\Phi_\omega :\R_+\rightarrow\R_+$ such that 
\begin{equation}
\label{eq:PLomega-dege}
\forall \theta \in \mathbb H,\qquad \omega(\theta) - \inf \omega \leqslant \Phi_\omega \po  |\C \na \omega(\theta)|^2\pf \,.
\end{equation}
\item Renormalised measure: there exists a non-decreasing $\widetilde{\Phi}:\R_+\rightarrow \R_+$ such that 
\begin{equation}\label{eq:ballisticLSI-dege}
\forall \rho \in \mathcal P_2(\mathbb H),\qquad  \frac1N \mathcal H(\rho|\tilde \nu_N) \leqslant \widetilde{\Phi} \po \re{\frac1{N^2}} \mathcal I_{\C}(\rho|\tilde \nu_N)\pf \,.
\end{equation}
\item  Gibbs measure: there exists a non-decreasing $\Phi_\infty:\R_+\rightarrow \R_+$ such that 
\begin{equation}\label{eq:unifLSI-dege}
\forall \rho^N \in \mathcal P_2(\R^{dN}),\qquad  \frac1N \mathcal H(\rho^N|\mu_\infty^N) \leqslant \Phi_\infty\po \frac1N \mathcal I(\rho^N|\mu_\infty^N)\pf \,,
\end{equation}
\end{enumerate}

The generalisation of the results from Section~\ref{sec:results}   reads as follows:

\begin{theorem}\label{thm:intermediaryimplications-dege}
Under Assumption~\ref{assum:baseLemme1}, the following holds:
\begin{enumerate}
\item If $\varphi$ is $\mathcal C^1$, Lipschitz continuous,  the $\Phi$-ŁI~\eqref{eq:PLF-dege} for $\mathcal F$ implies the $\widehat{\Phi}$-ŁI~\eqref{eq:PLFhat-dege} for $\hF$ with $\widehat{\Phi}(r) = \Phi \po \|\na \varphi\|_\infty^2 r \pf $.
\item   The $\widehat{\Phi}$-ŁI~\eqref{eq:PLFhat-dege} for $\hF$ implies the $\Phi_\omega$-ŁI~\eqref{eq:PLomega-dege} for $\omega$ with $\Phi_\omega(r) = \widehat{\Phi}(r)+r$.
\item   If $\omega$ satisfies the $\Phi_\omega$-ŁI~\eqref{eq:PLomega-dege} and is not constant nor Lipschitz, then necessarily $\Phi_\omega(r) > r$ for all $r>0$  and the $\widehat{\Phi}$-ŁI~\eqref{eq:PLFhat-dege} holds with $\widehat{\Phi}(r) = \Phi(r) - r$.
\item Under Assumption~\ref{assum:PL->unifLSI}, there exists a constant $K>0$ which depends only on the parameters $\mathfrak{P}$ in~\eqref{eq:param} such that, for any $N\geqslant 1$ if  the renormalised measures  $\tilde\nu_N$ satisfies the $\widetilde{\Phi}$-ŁI~\eqref{eq:ballisticLSI-dege} for some non-decreasing $\widetilde{\Phi}:\R_+\rightarrow \R_+$, then $\mu_\infty^N$ satisfies the $\Phi_\infty$-ŁI~\eqref{eq:unifLSI-dege} with $\Phi_\infty(r) = \widetilde{\Phi}(Cr) + Cr$.
\item Under Assumption~\ref{assum:unifLSI->PL}, if there exists a non-decreasing $\Phi_\infty:\R_+\rightarrow\R_+$ such that, for all $N\geqslant 1$ large enough, the Gibbs measure $\mu_\infty^N$ satisfies the $\Phi_\infty$-ŁI~\eqref{eq:unifLSI-dege}, then $\mathcal F$ satisfies the $\Phi$-ŁI~\eqref{eq:PLF-dege} with $\Phi = \Phi_\infty$.
\end{enumerate}
\end{theorem}

In Theorem~\ref{thm:intermediaryimplications-dege}, we haven't stated the analogue of Theorem~\ref{thm:ChewiStromme} (hence we don't have the complete analogue of the chain of implications~\eqref{eq:chaineimplications} in the degenerate case). Such a result would be nice, but establishing it exceeds the scope of the present work. Indeed, extending the arguments of~\cite{chewi2024ballistic} to the degenerate case doesn't seem straightforward (in particular, in the degenerate case, the Hessian of $\omega$ is singular at the minimiser, and the order of degeneracy may depend on the direction). Instead, we recall  \cite[Proposition 16]{Mtoymodel}, which provides a tractable way to establish a Low-temperature Łojasiewicz-LSI in degenerate convex cases.

\begin{proposition}[from~{\cite[Proposition 16]{Mtoymodel}}] \label{prop:critereconvexdege}
Under Assumption~\ref{assum:finitedim}, assume that  there exists $\beta\geqslant 2$ and $c_1,c_2>0$  such that, writing $\kappa(\theta) = \min(c_1,c_2|\C^{-1} \theta|^{\beta-2})$, for all $\theta,\zeta\in\mathbb H$,
\begin{equation}
\label{eq:convexiteomega-dege}
\forall t\in[0,1],\qquad t \omega(\theta) + (1-t) \omega(\zeta) \geqslant \omega (t\theta + (1-t)\zeta)  +   \kappa(\theta,\zeta) t(1-t) |\C^{-1} ( \theta-\zeta)|^2\,.
\end{equation}
Then there exists $C>0$ (depending only and explicitly on $c_1,c_2,\beta$ and $\dim \mathbb H$) such that for all $N\geqslant 1$, $\tilde \nu_N$ satisfies the $\widetilde{\Phi}$-ŁI~\eqref{eq:ballisticLSI-dege} with $\widetilde{\Phi}(r)  = C \min (r, r^{\frac{\beta}{2\beta-2}})$.
\end{proposition}

Only the case $\C=I$ is considered in \cite{Mtoymodel}, but is suffices to apply it to the preconditioned potential $\theta\mapsto \omega(\C\theta)$. The exponent $\frac{\beta}{2\beta-2}$ is sharp, as further discussed in~\cite{Mtoymodel}. For $\beta=2$, we retrieve the ballistic LSI~\eqref{eq:ballisticLSI}$+$\eqref{eq:ballisticlambda_N}.

\section{Proof for the uniform LSI}\label{sec:proofLSIN}

From~\eqref{eq:energy-quadra}, we decompose
\[U_N(\bx) = V_N(\bx) -  \frac 1{2N} \left|\sum_{i=1}^N\varphi(x_i)\right|^2\,\]
with
\[V_N(\bx) = \sum_{i=1}^N V(x_i) + N \mathcal E_c(\pi_{\bx})\,.\]
To decompose the Gibbs measure $\mu_\infty^N\propto \exp(-U_N)$, recalling $\rp = \C^{-1}\varphi$, we use that 
\begin{equation}
\label{eq:LaplaceGaussienne}
\mathbb E_{\gamma_{\C}} \po \exp\po \Theta\cdot  \sum_{i=1}^N \rp(x_i)\pf \pf  =  \exp\po \frac1{2N} \left| \sum_{i=1}^N \varphi(x_i)\right|^2\pf 
\end{equation}
where we recall that $\gamma_{\C}$ is  the centered Gaussian measure on $\mathbb H$ with covariance $\frac1N \C^2$. Hence,
\begin{equation}
\label{decomposition}
\int_{\R^{dN}} f(\bx) \mu_\infty^N(\dd \bx) = \int_{\mathbb H} \int_{\R^{dN}} f(\bx) \mu_{\theta}^N(\bx) \dd \bx \nu_N(\dd \theta)\,,
\end{equation}
where the microscopic measure is 
\begin{equation}
    \label{eq:def-mutheta-N}
    \mu_\theta^N(\bx) = \frac{1}{Z_N(\theta)} \exp\po - V_N(\bx) + \theta \cdot \sum_{i=1}^N \rp(x_i) \pf 
\end{equation}
with $Z_N(\theta)$ the normalizing constant and \re{the macroscopic fluctuation measure
\begin{equation}
\label{eq:nuNfluctuation}
\nu_N(\dd \theta) \propto  Z_N(\theta) \gamma_{\C}(\dd \theta)\,,
\end{equation}
 which makes sense since by Fubini and~\eqref{eq:LaplaceGaussienne}, 
\begin{equation}
\label{loc:sdfgdfgqmp}
\mathbb E_{\gamma_{\C}}\po  Z_N(\Theta)\pf  = \int_{E^N} e^{-U_N} <\infty\,. 
\end{equation}
}
We write
\begin{equation}
\label{eq:nuNfluctuation1}
 \hat \omega_N(\theta)  = - \frac{1}{N} \ln Z_N(\theta)\,,
\end{equation}
and, similarly to~\eqref{eq:defomega},
\begin{equation}
\label{eq:nuNfluctuation2}
 \omega_N(\theta)  = \hat \omega_N(\theta) + \frac12 |\C^{-1}\theta|^2\,.
\end{equation}
Notice that, when $\dim \mathbb H<\infty$, $\nu_N \propto e^{-N\omega_N}$.

The next lemma is an adaptation of \cite[Proposition 4.5]{Dagallier} in our slightly more general results.

\begin{lemma}\label{lem:microscopicLSI}
Under Assumption~\ref{assum:PL->unifLSI},   there exists $c,N_0>0$, depending only on the parameters $\mathfrak{P}$ in~\eqref{eq:param}, such that, for any $\theta=(\theta_\ell,\theta_b)\in\mathbb H$, $\mu_{\theta}^N$ satisfies a LSI with constant $\lambda_{N,\theta}$ satisfying
\[\frac{1}{\lambda_{N,\theta}}\leqslant \left\{\begin{array}{ll}
c e^{c |\theta_b|} & \quad \forall N \geqslant N_0 e^{c|\theta_b|} \\
c e^{c(N + |\theta_b|)} & \quad \forall N\geqslant 1.
\end{array}\right.\]

\end{lemma}

\begin{proof}
Under Assumption~\ref{assum:PL->unifLSI}, the first point is a consequence of  \cite[Theorem 1]{SongboLSI}. Actually, we will apply \cite[Corollary 3]{M61}  (which is based on the former) where, in the flat-convex case, the conditions are stated in a  slightly different way.  
\begin{itemize}
\item The energy $\rho \mapsto \mathcal E_\theta(\rho)= \mathcal E_0(\rho) - \theta\cdot \rho(\rp)$ is flat-convex (since $\rho \mapsto \rho(V) - \rho(\rp)$ is linear).
\item The flat derivatives of $\mathcal E_\theta$ are
\[
\frac{\delta \mathcal E_\theta}{\delta m}(\nu,x) = V(x) + \frac{\delta \mathcal E_c}{\delta m}(\nu,x) - \theta\cdot \rp(x)\,,\qquad
\frac{\delta^2 \mathcal E_\theta}{\delta m^2}(\nu,x) = \frac{\delta^2 \mathcal E_c}{\delta m^2}(\nu,x)\,.
\]
The conditions of Assumption~\ref{assum:PL->unifLSI} implies that these functions are $\mathcal C^2$ in $x$, and that $\na_x \frac{\delta^2 \mathcal E_\theta}{\delta m^2}$ and   $D^2\mathcal E_\theta $ are bounded (independently from $\theta$).
\item For any $\nu \in \mathcal P_2(E)$, $\frac{\delta \mathcal E_\theta}{\delta m}(\nu,\cdot)$ is the sum of    $V_c - \theta_\ell \cdot \rp_\ell$ (strongly convex with a constant independent from $\theta$), $V_\ell$ (Lipschitz continuous) and $V_b -  \theta_b \cdot \rp_b + \frac{\delta \mathcal E_c}{\delta m}(\nu,\cdot)$ (bounded by $\|V_b\|_\infty+|\theta_b|\|\rp_b\|_\infty+\|\frac{\delta \mathcal E_c}{\delta m}\|_\infty$). By the classical Bakry-Emery, Holley-Stroock and Aida-Shigekawa criteria for LSI~\cite{BakryGentilLedoux,HolleyStroock,AidaShigekawa,cattiaux2022functional}, we get that there exists $c>0$, independent from $\nu$ and $\theta$, such that $\mu \propto \exp(-\frac{\delta \mathcal E_\theta}{\delta m}(\nu,\cdot))$ satisfies a LSI with constant $c e^{c |\theta_b|}$.

Moreover, since $\na_x  \frac{\delta^2 \mathcal E_c}{\delta m^2}$ is bounded (independently from $\theta$), as discussed in \cite[Section 4]{SongboLSI},  for any $N\geqslant 2$ and $\bx_{\neq 1}=(x_i)_{i\in\cco 2,N\ccf}$, the probability measure $m_{\bx_{\neq 1}}$ with density proportional to $x_1 \mapsto \exp(-N\mathcal E_\theta(\pi_{\bx}))$ is obtained from $x_1\mapsto \exp(-\frac{\delta \mathcal E_\theta}{\delta m}(\pi_{\bx_{\neq 1}}))$ by a log-Lipschitz perturbation with Lipschitz constant independent from $N$ and $\bx_{\neq 1}$. As a consequence, by the Aida-Shigekawa criterion, there exists $c>0$ such that $m_{\bx_{\neq 1}}$ satisfies a LSI with constant $c e^{c|\theta_b|}$ for all $N\geqslant 2$ and $\bx_{\neq 1} \in E^{N-1}$. 
\end{itemize}

All these conditions show that \cite[Corollary 3]{M61} applies and gives for $\mu_\theta^N$ as LSI with constant $c e^{c|\theta_b|}$ for all $N \geqslant c e^{c|\theta_b|}$ for some constant $c$ depending only on the parameters $\mathfrak{P}$, which is the first statement of the lemma.

 For the second one, we see that
\[V_N(\bx) - \theta \cdot \sum_{i=1}^N \rp(x_i) = \sum_{i=1}^N \co V(x_i) - \theta_\ell \cdot \rp_\ell(x_i) - \theta_b \cdot \rp_b(x_i)\cf + N \mathcal E_c(\pi_{\bx})  \,. \]
The potential $V - \theta_\ell \cdot \rp_\ell - \theta_b \cdot \rp_b $ is the sum of $V_c - \theta_\ell \cdot \rp_\ell$ (strongly convex with a constant independent from $\theta$), $V_\ell$ (Lipschitz continuous) and $V_b -  \theta_b \cdot \rp_b$ (bounded by $\|V_b\|_\infty+|\theta_b|\|\rp_b\|_\infty$). Again, by the Bakry-Emery, Holley-Stroock and Aida-Shigekawa criteria, and the tensorisation property of LSI, the measure with density proportional to
\[\exp\po -\sum_{i=1}^N \co V(x_i) - \theta_\ell \cdot \rp_\ell(x_i) - \theta_b \cdot \rp_b(x_i)\cf \pf\]
satisfies a LSI with constant $c e^{c(1+|\theta_b|)}$ for some $c>0$ depending only on $\mathfrak{P}$. Then, by the Holley-Stroock result,  $\mu_\theta^N$ satisfies a LSI with constant $c e^{c(1+|\theta_b|)+N\|\mathcal E_c\|_\infty}$, which concludes the proof.
\end{proof}

The fact that the bounds in Lemma~\ref{lem:microscopicLSI} are not uniform in $N$ and $\theta$ will not be a problem, as we will need an integrated bound as follows:

\begin{lemma}\label{lem:gammaX}
Under Assumption~\ref{assum:PL->unifLSI}, there exists $M>0$, depending only on the parameters $\mathfrak{P}$, such that for all $N\geqslant 1$ and  $\bx\in E^N$, denoting by $\gamma_{\bx}$ the Gaussian density on $\mathbb H$ with mean $\pi_{\bx}(\rp)$ and variance $\frac1N \C^2$,
\[\mathbb E_{\gamma_{\bx}} \po \frac{1}{\lambda_{N,\Theta}^2}\pf  \leqslant M\,.\]
\end{lemma} 
\begin{proof}
Let $A>0$, to be chosen large enough later on. Thanks to Lemma~\ref{lem:microscopicLSI}, distinguishing whether $N \lessgtr N_0 e^{cA}$,
\[\mathbb E_{\gamma_{\bx}} \po \frac{1}{\lambda_{N,\Theta}^2}\1_{|\Theta_b|\leqslant A} \pf     \leqslant c^2 e^{2cA} + c^2e^{2c N_0 e^{cA}  + cA} \,.\]
Taking $A>2\|\rp_b\|_\infty$, we get that $|\theta_b|>r$ implies that $|\theta_b - \pi_{\bx}(\rp_b)|> r/2$ for all $r\geqslant A$, from which, by Gaussian concentration (see e.g. \cite[Equation (3.5)]{ledoux1991isoperimetry}),
\[\mathbb P_{\gamma_\bx} \po |\Theta_b|>r \pf   \leqslant 4 \exp\po - \frac{Nr^2}{8\mathrm{Tr}(\C) } \pf\,.  \]
Using again Lemma~\ref{lem:microscopicLSI},
\begin{align*}
\mathbb E_{\gamma_{\bx}} \po \frac{1}{\lambda_{N,\Theta}^2}\1_{|\Theta_b|> A} \pf    &  \leqslant c^2 e^{2cN} \mathbb E_{\gamma_{\bx}} \po e^{2c|\Theta_b|}\1_{|\Theta_b|> A} \pf  \\
& =  2c^3 e^{2cN}  \int_{A}^\infty \mathbb P \po  |\Theta_b| \geqslant r \pf   e^{2cr} \dd r \\
& \leqslant 8c^3 e^{2cN} \int_{A}^\infty \frac{r}{A} \exp\po - \frac{Nr^2}{8\mathrm{Tr}(\C)} + 2cr\pf \dd r \\
&\leqslant  \frac{32c^3}A     \exp\po 2cN - \frac{NA^2}{8\mathrm{Tr}(\C)} + 2cA\pf \,.
\end{align*}
Taking $A\geqslant 4\sqrt{c\mathrm{Tr}(\C)}$, this is bounded uniformly in $N$, which concludes the proof.
\end{proof}

Next, we turn to the study of $\nu_N$ in~\eqref{eq:nuNfluctuation}. By the Laplace-Varadhan lemma, we expect   $\omega_N$ to converge as $N\rightarrow \infty$ to $\omega$ given in~\eqref{eq:defomega}. Actually we have the following quantitative result (corresponding to \cite[Proposition 4.7]{Dagallier}).

\begin{lemma}\label{lem:omegaNtoomega}
Under Assumption~\ref{assum:PL->unifLSI}, for all $N\geqslant 1$,
\begin{equation}
\label{eq:omegaNtoomega}
 \| \hat \omega(\theta) - \hat \omega_N(\theta)\|_\infty \leqslant \frac1N \left\|\frac{\delta^2 \mathcal E_c}{\delta m^2}\right\|_\infty\,.
\end{equation}
\re{In particular, the probability measure $\tilde \nu_N$ with density proportional to $\exp(-N\hat \omega(\theta) )$ with respect to $\gamma_{\C}$  is well defined.}
\end{lemma}

\begin{proof}
Recall the self-consistency equation~\eqref{eq:selfconstitenttheta} solved by $\bar\mu_\theta$. 
 Write $Z_\theta$ the normalising constant. Hence, using the representation~\eqref{eq:omegaegalite} of $\hat \omega$,
 \begin{align*}
\hat \omega(\theta) &= \int \bar\mu_\theta\ln\bar\mu_\theta + \mathcal E_0(\bar\mu_\theta) - \theta\cdot \bar\mu_\theta(\rp)     \\
&=    - \ln Z_\theta - \bar\mu_\theta \po  \frac{\delta\mathcal E_0}{\delta m}(\bar \mu_\theta)\pf + \mathcal E_0(\bar\mu_\theta) \,.
 \end{align*}
 Hence
 \begin{align*}
 \hat \omega(\theta) -\hat  \omega_N(\theta) &= \frac1N \ln Z_N(\theta)  - \ln Z_\theta - \bar\mu_\theta \po  \frac{\delta\mathcal E_0}{\delta m}(\bar \mu_\theta)\pf + \mathcal E_0(\bar\mu_\theta) \\
 &= \frac1N \ln \po \frac{Z_N(\theta) \exp\po -N  \bar\mu_\theta \po  \frac{\delta\mathcal E_0}{\delta m}(\bar \mu_\theta)\pf + N \mathcal E_0(\bar\mu_\theta)\pf  }{Z_\theta^N} \pf \\
 &= \frac1N \ln \int_{E^{N}} \exp\po - N \co \mathcal E_0(\pi_{\bx}) +   (\bar\mu_\theta- \pi_{\bx}) \po  \frac{\delta\mathcal E_0}{\delta m}(\bar \mu_\theta)\pf - \mathcal E_0(\bar\mu_\theta) \cf \pf \bar\mu_\theta^{\otimes N}(\dd \bx)\,.
 \end{align*}
 The linear part of $\mathcal E_0$ cancels out in this expression, so we can replace $\mathcal E_0$ by $\mathcal E_c$. Since $\mathcal E_c$ is convex, the term in brackets is non-negative, hence $ \hat \omega(\theta) -\hat  \omega_N(\theta) \leqslant 0$.  

 Moreover, by Jensen,
  \begin{align*}
 \hat \omega(\theta) - \hat  \omega_N(\theta) &\geqslant   - \int_{E^{N}}  \co \mathcal E_c(\pi_{\bx}) +   (\bar\mu_\theta- \pi_{\bx}) \po  \frac{\delta\mathcal E_c}{\delta m}(\bar \mu_\theta)\pf - \mathcal E_c(\bar\mu_\theta) \cf \bar\mu_\theta^{\otimes N}(\dd \bx) \\
 &= - \frac12 \int_{E^{N}}  \int_0^t \int_{E^2} g_s \dd \po \pi_{\bx} - \bar\mu_{\theta}\pf^{\otimes 2} \dd s \bar\mu_\theta^{\otimes N}(\dd \bx)
 \end{align*}
 with
 \[g_s(x,y) = \frac{\delta^2 \mathcal E_c}{\delta m^2}\po s\pi_{\bx} +(1-s) \bar\mu_{\theta},x,y\pf \,.\]
 This gives~\eqref{eq:omegaNtoomega} since
 \[ \int_{E^{N}}   \int_{E^{2}} g_s \dd \po \pi_{\bx} - \bar\mu_{\theta}\pf^{\otimes 2}  \bar\mu_\theta^{\otimes N}(\dd \bx) = \frac1N \int_{E} g_s(x,x)\bar\mu_\theta(\dd x) - \frac1N \int_{E^2} g_s \bar\mu_\theta^{\otimes 2}\,.\]
\re{
The conclusion for $\tilde \nu_N$ follows from~\eqref{loc:sdfgdfgqmp}, since
\[\mathbb E_{\gamma_{\C}} \po e^{-N \po \omega(\Theta) - \frac12\|\Theta\|_{\C}^2 \pf }\pf \leqslant e^{\left\|\frac{\delta^2 \mathcal E_c}{\delta m^2}\right\|_\infty} \mathbb E_{\gamma_{\C}} \po e^{-N \po \omega_N(\Theta) - \frac12\|\Theta\|_{\C}^2 \pf }\pf = e^{\left\|\frac{\delta^2 \mathcal E_c}{\delta m^2}\right\|_\infty} \int_{E^N} e^{-U_N} < \infty\,.   \] 
}
\end{proof}

Via the Holley-Stroock perturbation criterion for LSI, a direct consequence of Lemma~\ref{lem:omegaNtoomega} is the following:

\begin{lemma}\label{lem:LISnuN}
Under Assumption~\ref{assum:PL->unifLSI},  for any $N\geqslant 1$, if the normalised measure $\tilde \nu_N$ satisfies a LSI~\eqref{eq:ballisticLSI}   then the fluctuation measure  $\nu_N$ in~\eqref{eq:nuNfluctuation} satisfies the same LSI with constant $e^{-a} \tilde \lambda_N$ with $a=\|\frac{\delta^2 \mathcal E_c}{\delta m^2}\|_\infty$.
\end{lemma}

\begin{proof}[Proof of Theorem~\ref{thm:PL->unifLSI}]
From Lemmas~\ref{lem:gammaX} and \ref{lem:LISnuN}, we simply have to follow the proof of \cite[Proposition 4.6 and Theorem 4.2]{Dagallier}, that we recall here for completeness. Take $\nu' \in\mathcal P_2(E^{N})$ with relative density $f=\dd \nu'/\dd \mu_\infty^N$ and write $F(\theta) = \int_{E^N} f(\bx) \mu_{\theta}^N(\bx)\dd \bx$. Based on the representation~\eqref{decomposition}, $\nu_M'(\dd \theta):=F(\theta)\nu_N(\dd \theta)$ is the marginal density (in $\theta$) of $ \nu_m'(\dd \bx,\dd \theta):=f(\bx)\mu_{\theta}^N(\bx)\dd \bx\nu_N(\dd \theta)$, and $f(\bx)/F(\theta)$ is the conditional density (in $\bx$ given $\theta$) of $\nu_m'$ with respect to $\mu_\theta^N$.  The classical macro-micro decomposition of the entropy gives
 \begin{align}
 \mathcal H\po \nu|\mu_\infty^N \pf &= \int_{\mathbb H} \int_{E^N} f \ln f \dd \mu_\theta^N \nu_N(\dd \theta )\nonumber\\
 &= \int_{\mathbb H} F \ln F \dd \nu_N + \int_{\mathbb H} \co \int_{E^N} \frac{f}{F(\theta)}\ln \frac{f}{F(\theta)}  \mu_\theta^N\cf F(\theta) \nu_N( \dd \theta) \,. \label{eq:micromacroentropie}
 \end{align}
 The microscopic LSI from Lemma~\ref{lem:microscopicLSI} gives, for all $\theta\in \mathbb H$,
\[
 \int_{E^N} \frac{f}{F(\theta)}\ln \frac{f}{F(\theta)}  \mu_\theta^N   \leqslant \frac1{2\lambda_{N,\theta}} \int_{E^N}\left|\na_{\bx} \ln \frac{f}{F(\theta)}\right|^2   \frac{f}{F(\theta)}  \mu_\theta^N  = \frac1{2\lambda_{N,\theta}} \int_{E^N}\left|\na_{\bx} \ln f \right|^2   \frac{f}{F(\theta)}  \mu_\theta^N  \,.
\] 
Integrating with respect to $\theta$ gives
\begin{align}
\int_{\mathbb H} \co \int_{E^N} \frac{f}{F(\theta)}\ln \frac{f}{F(\theta)}  \mu_\theta^N\cf F(\theta) \nu_N(\dd \theta) 
& \leqslant  \int_{\mathbb H} \frac1{2\lambda_{N,\theta}} \int_{E^N}\left|\na_{\bx} \ln f\right|^2  f   \mu_\theta^N \nu_N(  \dd \theta) \nonumber\\
&=  \int_{E^N}\left|\na_{\bx} \ln f\right|^2  f \co \int_{\mathbb H} \frac1{2\lambda_{N,\theta}} \frac{\mu_\theta^N }{\mu_\infty^N}\nu_N(\dd \theta) \cf \mu_\infty^N  \,.\label{loc:azazerzetlm}
\end{align}
From~\eqref{decomposition}, for all $\bx\in E^N$, $\mu_\theta^N(\bx) \nu_N(\dd\theta) / \mu_\infty^N(\bx)$ is a conditional probability, equal to  
\begin{align*}
\frac{\mu_\theta^N(\bx) }{\mu_\infty^N(\bx)} \nu_N(\dd \theta) &  \propto  \exp\po - V_N(\bx) + \theta \cdot \sum_{i=1}^N \rp(x_i)   + U_N(\bx) \pf \gamma_{\C}(\dd \theta)
 \propto \gamma_{\pi_{\bx}}(\dd \theta)
\end{align*}
introduced in Lemma~\ref{lem:gammaX}. Using this lemma in~\eqref{loc:azazerzetlm} (with $2 \lambda_{N,\theta}^{-1} \leqslant \lambda_{N,\theta}^{-2}+1$) gives
\begin{equation}
\label{loc:qfsffa}
 \int_{\mathbb H} \co \int_{E^N} \frac{f}{F(\theta)}\ln \frac{f}{F(\theta)}  \mu_\theta^N\cf F(\theta) \nu_N(\dd \theta)  \leqslant \frac{M+1}{4} \mathcal I(\nu|\mu_\infty^N)\,.
\end{equation}

We now turn to the study of the macroscopic term in~\eqref{eq:micromacroentropie}. The LSI  from Lemma~\ref{lem:LISnuN} gives  
\begin{equation}
\label{loc:vdfgemfa}
\int_{\mathbb H} F \ln F \dd \nu_N  \leqslant \frac{4}{\kappa N} \int_{\mathbb H} |\C\na_\theta \sqrt{ F}|^2  \dd \nu_N \,.
\end{equation}
Recalling the expression~\eqref{eq:def-mutheta-N} of $\mu_\theta^N$,
\begin{align*}
2 \C \na_\theta \sqrt{ F}(\theta) &= \frac{1}{\sqrt{F(\theta)}} \C \na_\theta F(\theta) \\
& = \frac{1}{\sqrt{F(\theta)}} \int_{E^N} f(\bx) N \po \pi_{\bx}( \varphi) - \int_{E^N} \pi_{\mathbf{y}}(\varphi)\mu_\theta^N(\mathbf{y})\dd \mathbf{y}\pf \mu_\theta^N (\bx)\dd \bx\\
&= \frac{N}{\sqrt{F(\theta)}} \mathrm{Cov}_{\mu_\theta^N}\po f(\bX),\pi_{\bX}(\varphi)\pf \,.
\end{align*}
For any $\zeta\in \mathbb H$ with $|\zeta|=1$, $\bx \mapsto  \zeta \cdot \pi_{\bx} (\varphi) $ is $N^{-1/2}\|\na \varphi\|_\infty$-Lipschitz, and thus \cite[Lemma A.1]{Dagallier} gives
\begin{equation}
\label{loc:dmlpaldfdg}
2 \zeta\cdot  \C \na_\theta \sqrt{ F}(\theta) = \frac{N}{\sqrt{F(\theta)}} \mathrm{Cov}_{\mu_\theta^N}\po f(\bX), \zeta\cdot \pi_{\bX}(\varphi)\pf \leqslant \frac{2\sqrt{N}\|\na \varphi\|_\infty}{ \lambda_{N,\theta}} \sqrt{ \mu_{\theta}^N(|\na \sqrt{f}|^2)} 
\end{equation}
(we used that $\mu_\theta^N(f) = F(\theta)$). Since $\zeta$ is arbitrary, this gives a bound on $|\C \na_\theta \sqrt{F}|$, which we plug in~\eqref{loc:vdfgemfa} to get that
\begin{align*}
\int_{\mathbb H} F \ln F \dd \nu_N  & \leqslant \frac{8\|\na \varphi\|_\infty^2}{\kappa } \int_{\mathbb H}\int_{E^N} \frac{1}{\lambda_{N,\theta}^2} |\na_{\bx} \sqrt{ f}|^2  \dd \mu_{\theta}^N\dd \nu_N \,.
\end{align*} 
Reasoning as we did from~\eqref{loc:azazerzetlm} gives
\[\int_{\mathbb H} F \ln F \dd \nu_N  \leqslant \frac{2\|\na \varphi\|_\infty^2M}{\kappa }\mathcal I(\nu|\mu_\infty^N)\,.\]
Combining this and~\eqref{loc:qfsffa} in~\eqref{eq:micromacroentropie} concludes the proof.
\end{proof}

\section{Other proofs}\label{sec:otherproofs}

\subsection{Coarse-graining}

Before establishing the various results stated in Section~\ref{sec:results}, we start with a useful lemma which clarifies the relations between $\mathcal F$, $\hF$ and $\omega$. Recall the definition~\eqref{eq:defFtheta} of $\mathcal F_\theta$.

\begin{lemma}\label{lem:omegaFhat}
Assume that for all $\theta \in \mathbb H$, $\mathcal F_\theta$
admits over $\mathcal P_2(E)$ a unique global minimiser $\bar\mu_\theta \in \mathcal P_2(E)$. Then:
\begin{enumerate}
\item For all $\theta\in \mathbb H$,  $\bar\mu_\theta  $ solves the so-called self-consistency equation
\begin{equation}
\label{eq:selfconstitenttheta}
 \bar\mu_\theta \propto \exp\po - \frac{\delta \mathcal E_0}{\delta m}(\bar \mu_\theta,\cdot  ) + \theta \cdot  \rp \pf \,,
\end{equation} 
and it is characterised by $m_\theta := \bar\mu_\theta(\rp)$ in the sense that $\bar\mu_\theta$ is the unique minimiser of $\mathcal F$ among $\{\rho\in\mathcal P_2(E),\ \rho(\rp)=m_\theta\}$. 
 Moreover, $\bar \mu_\theta(\varphi)$ is the unique minimiser on $\mathbb H$ of $m \mapsto \hF(m) + \frac12|\C m|^2 - 2 m\cdot \theta$, and in particular $\hF(m_\theta)<\infty$ for all $\theta\in\mathbb H$.  For all $\theta\in\mathbb H$,
\begin{equation}\label{eq:omegaegalite}
\hat \omega(\theta) = \mathcal F(\bar\mu_\theta) + \frac12|\C m_\theta|^2 - m_\theta\cdot \theta  = \hF(m_\theta)+ \frac12|\C m_\theta|^2 - m_\theta\cdot \theta\,.
\end{equation}
\item Enforcing furthermore Assumption~\ref{assum:finitedim} and  that the regularity condition~\eqref{eq:regularitymutheta} holds, 
 then $\omega$ is $\mathcal C^1$ over $\mathbb H$ and $\hF$ is $\mathcal C^1$ over $\mathcal M=\{m_\theta,\ \theta\in \mathbb H\}$, and for all $\theta \in \mathbb H$,
\begin{equation}
\label{eq:naomega}
 \C^2 \na \omega(\theta) = \theta - \C^2 m_\theta =   \na \hF(m_\theta) \,.
\end{equation}
Moreover, the set $\mathcal K= \{\theta\in\mathbb H,\ \theta= \C^2 m_\theta\}$ is exactly the set of critical points of $\omega$ and $\hF$, and $\theta \mapsto \bar\mu_\theta$ is a bijection from $\mathcal K$ to the set of critical points of $\mathcal F$ (with inverse map $\mu \mapsto  \mu(\varphi)$). It is also a bijection between the global minimisers of $ \omega$ (which are exactly the global minimisers of $\hF$) and those of $\mathcal F$. Besides, $\inf \mathcal F= \inf \omega = \inf \hF$.

\end{enumerate}
\end{lemma}

\begin{remark}[concerning the assumptions in Lemma~\ref{lem:omegaFhat}]\label{rem:regularityPDE}
We provide in Lemma~\ref{lem:uniqueminimumtheta} below  a simple condition under which $\mathcal F_\theta$ admits a unique minimiser, and in Lemma~\ref{lem:regularitycheck} a condition  to check the regularity condition~\eqref{eq:regularitymutheta} when $\mathcal E_c$ is also a function of the parameter $\varphi$. In more general cases, the minimiser is characterised as the solution of the non-linear elliptic equation
\[\Delta \bar\mu_\theta + \na\cdot \po D\mathcal E_\theta(\bar\mu_\theta,\cdot) \bar\mu_\theta \pf = 0\,,\]
with the modulated energy
\begin{equation}
\label{eq:Etheta}
\mathcal E_\theta(\rho) = \mathcal E(\rho) + \frac12|\rho(\varphi)|^2 - \rho(\rp)\cdot \theta  \,. 
\end{equation}
The regularity of $\bar\mu_\theta$ (hence of $\mathcal F(\bar\mu_\theta)$ and $\bar\mu_\theta(\rp)$ as required in~\eqref{eq:regularitymutheta}) can thus be established under suitable conditions on $D\mathcal E_\theta$ following the stability theory of solutions of elliptic equations with respect to parameters. Moreover, it may be possible to remove the regularity conditions~\eqref{eq:regularitymutheta}  and to replace gradients in~\eqref{eq:naomega} by slopes or upper gradients (see \cite[Chapter 1]{ambrosio2005gradient}). Indeed, the only information we need in our analysis is  that $|\C\na \omega(\theta)| =| \C^{-1} \theta -\C m_\theta| = |\C^{-1}\na \hF(m_\theta)|$. Finally, when $\dim\mathbb H<\infty$ and $\C=I_{\mathbb H}$ for instance, let us notice that $\omega$ being a Moreau envelope of $\hF$, then it is automatically $\mathcal C^1$ and $\na \omega(\theta)=\theta-m_\theta$ if $\hF$ is $1$-semi-convex, see \cite[Proposition 5]{renaud2025moreau}.

  We will not discuss this further in order to focus on the rest of the analysis, keeping the regularity condition~\eqref{eq:regularitymutheta} as a general assumption. 
\end{remark}

\begin{remark}
When $\dim\mathbb H<\infty$ and $\C=I_{\mathbb H}$, since~\eqref{eq:defomega} shows that $\omega$ is a Moreau envelope of $\hF$, then local minimisers of $\omega$   correspond to local minisers of $\hF$, cf. \cite[Theorem 3.5]{khanh2025local}. With a linear change of variables, we deduce the same result when $\C\neq I_{\mathbb H}$.
\end{remark}

\begin{proof}[Proof of Lemma~\ref{lem:omegaFhat}]
As a global minimizer, $\bar\mu_\theta$ is a critical point of $\mathcal F_\theta$, namely $\frac{\delta \mathcal F_\theta}{\delta m}(\bar\mu_\theta,\cdot) = \frac{\delta \mathcal E_\theta}{\delta m}(\bar\mu_\theta,\cdot)  + \ln \bar\mu_\theta = constant$ (recall that $\frac{\delta \mathcal F_\theta}{\delta m} $ is defined up to an additive constant), which amounts to~\eqref{eq:selfconstitenttheta}.

Let $\rho \neq \bar\mu_\theta$ be such that $\rho(\varphi) = \C m_\theta$ (i.e. $\rho(\rp) = m_\theta = \bar\mu_\theta(\rp)$). Since
\[\mathcal F_\theta(\rho) > \mathcal F_\theta(\bar\mu_\theta)\,,\]
and in view of the definition~\eqref{eq:defFtheta} of $\mathcal F_\theta$ we equivalently get that
\[\mathcal F(\rho) > \mathcal F(\bar\mu_\theta)\,.\]
In other words, $\bar\mu_\theta$ is the unique minimiser of $\mathcal F$ among $\{\rho\in\mathcal P_2(E), \rho(\rp)=m_\theta\}$. As a consequence, $\bar\mu_\theta$ is characterised by $m_\theta$, and it is such that
\begin{equation}
\label{loc:sdfsdgsd}
\hF(m_\theta) = \mathcal F(\bar \mu_\theta) \,.
\end{equation}
Recalling the definition~\eqref{eq:defomega1} of $\hat \omega$ and that $\bar\mu_\theta$ minimises $\mathcal F_\theta$, this gives
\begin{equation}
\label{eqloc:omegatheta1}
\hat \omega (\theta) = \mathcal F(\bar\mu_\theta) + \frac12\re{|\C m_\theta|^2} - m_\theta\cdot \theta  =  \hF(m_\theta) + \frac12\re{|\C m_\theta|^2}  - m_\theta\cdot \theta \,.
\end{equation}
This concludes the proof of the first part of the lemma. From now on, Assumption~\ref{assum:finitedim} and~\eqref{eq:regularitymutheta} holds.

As a consequence of~\eqref{eqloc:omegatheta1}, thanks to~\eqref{eq:regularitymutheta}, $\omega$ (resp. $\hF$) is $\mathcal C^1$ over $\mathbb H$ (resp. $\{m_\theta : \theta\in\mathbb H\}$).  Moreover,~\eqref{loc:sdfsdgsd} means that the second infimum in~\eqref{eq:defomega1} is  attained at $m_\theta$, and the critical point equation gives
\[\na \hF(m_\theta) = \theta - \C^2 m_\theta\,. \]
Finally, taking the derivative in~\eqref{eqloc:omegatheta1}, 
\begin{equation*}
 \na \omega(\theta) = \C^{-2} \theta + \na \hat \omega (\theta) = \C^{-2} \theta + \na m_\theta  \cdot \po  \na \hF(m_\theta) +(\C^2 m_\theta- \theta) \pf  \re{ - m_\theta} = \re{\C^{-2} (\theta - \C^2 m_\theta)}\,.
\end{equation*}
These two last equalities show that the points in $\mathcal K=\{\theta\in\mathbb H,\ \theta=\C^2 m_\theta\}$ are exactly the critical points of $\omega$ and of $\hF$. Moreover, for $\theta\in \mathcal K$, the self-consistency equation~\eqref{eq:selfconstitenttheta} becomes
\begin{equation}
 \bar\mu_\theta \propto \exp\po - \frac{\delta \mathcal E_0}{\delta m}(\bar \mu_\theta ) + \C^2 \bar\mu_\theta(\rp) \cdot  \rp \pf = \exp\po - \frac{\delta \mathcal E_0}{\delta m}(\bar \mu_\theta ) + \bar\mu_\theta(\varphi) \cdot  \varphi \pf \,,
\end{equation}  
which is exactly the self-consistency equation which characterises the critical points of $\mathcal F$. Conversely, any  critical point $\mu_*$ of $\mathcal F$ satisfies
\[ \mu_* \propto \exp\po - \frac{\delta \mathcal E_0}{\delta m}(\mu_* ) + \mu_*(\varphi) \cdot  \varphi \pf \,.
\] 
Letting $\theta= \C^2 \mu_*(\rp)$, we see that $\mu_*$ satisfies~\eqref{eq:selfconstitenttheta}, and thus $\mu_* = \bar\mu_\theta$ by uniqueness of the critical point of $\mathcal F_\theta$.

Concerning global minimisers: by the definition~\eqref{eq:defomega}, $\omega(\theta) \geqslant \inf\mathcal F$ for all $\theta\in\mathbb H$. If $\mu_*$ is a global minimiser of $\mathcal F$ then it is a critical point and thus $\theta_*:=  \C \mu_*(\varphi)$ satisfies $\theta_* = \C^2 m_{\theta_*}$, so that $\omega(\theta_*) = \mathcal F(\mu_*) = \inf\omega$. With the same argument, from the second representation of $\omega$ in~\eqref{eq:defomega1}, we get that a  minimiser of $\hF$ is a minimiser of $\omega$. Similarly, from~\eqref{eq:defFhat}, $\hF(\theta) \geqslant \inf\mathcal F$ for all $\theta$, and the equality holds at $\theta_*= \C \mu_*(\varphi)$ with $\mu_*$ a minimiser of $\mathcal F$. Finally, let us show that if $\theta$ is a minimiser of $\omega$, then $\bar \mu_\theta$ is a minimiser of $\mathcal F$. From~\eqref{eq:defomega},
\[\inf \omega = \inf_{\rho\in \mathcal P_2(E)} \inf_{\theta\in\mathbb H} \left\{\mathcal F(\rho) + \frac{1}{2}|\rho(\varphi)-\theta|^2 \right\} = \inf \mathcal F \,.\]
If $\theta$ is a minimiser of $\omega$, by the variational definition of $\bar\mu_\theta$ and the fact that $\theta= \C^2 m_\theta = \C\bar\mu_\theta(\varphi)$,
\[\inf\mathcal F = \inf \omega = \omega(\theta) = \mathcal F(\bar\mu_\theta) + \frac12 |\bar\mu_\theta(\varphi)-\C^{-1}\theta|^2 = \mathcal F(\bar\mu_\theta)\,,\]
which concludes.
\end{proof}

\begin{lemma}\label{lem:uniqueminimumtheta}
Assume that $\mathcal E$ is of the form~\eqref{eq:energy-quadra} with $\mathcal E_c$ flat-convex, that we can decompose $\psi=(\psi_\ell,\psi_b)$ where $\psi_\ell$ is linear and $\psi_b$ is bounded (simply $\psi = \psi_b$ is bounded if $E=\T^d$), and that  there exists $C,c\geqslant  0$ such that, for any $\mu\in\mathcal P_2(E)$,
\begin{equation}
\label{loc:fztmlmaa}
\mu(V) + \mathcal E_c(\mu ) \geqslant c \int_{E} |x|^2\mu(\dd x) - C 
\end{equation}
(with $c=0$ if $E=\T^d$). Then, for all $\theta\in \mathbb H$, $\mathcal F_\theta$ in~\eqref{eq:defFtheta} admits a unique minimiser $\bar\mu_\theta\in\mathcal P_2(E)$.
\end{lemma}

\begin{proof}[Proof of Lemma~\ref{lem:uniqueminimumtheta}]
With the form~\eqref{eq:energy-quadra}, for any $\theta\in\mathbb H$, the energy~\eqref{eq:Etheta} is
\[\mathcal E_\theta(\rho)  = \rho(V) + \mathcal E_c(\rho) -  \theta\cdot \rho(\rp) \,, \]
and it is flat-convex. Since the entropy is strictly flat-convex, so is $\mathcal F_\theta  = \mathcal E_\theta + \mathcal H$. When $E=\T^d$, from~\eqref{loc:fztmlmaa}, the bound on $\psi$ and the fact that the entropy is bounded below by the one of the uniform measure, we get that $\mathcal F_\theta$ is lower bounded. When $E=\R^d$, similarly,
\begin{align*}
\mathcal F_\theta(\rho )& \geqslant \frac c2 \int_{\R^d} |x|^2\rho(\dd x) - C - \|\rp_b\|_\infty |\theta|_b - \frac{1}{2c} |\theta_\ell|^2 \|\na \rp_\ell \|_\infty^2 + \mathcal H(\rho) \\
& \geqslant  \frac c4  \int_{\R^d} |x|^2\rho(\dd x) +  \mathcal H \po \rho | \mathcal N(0,c^{-1} I_d)\pf - C_\theta'
\end{align*}
for some constant $C_\theta'$. Hence $\mathcal F_\theta$ is lower-bounded and coercive. In both cases $E\in\{\T^d,\R^d\}$, by lower semi-continuity (implied by the flat convexity), it reaches its infimum at some $\bar\mu_\theta \in \mathcal P_2(E)$, which is unique by strong convexity. 
\end{proof}

\new{
We conclude with a situation where we can check the regularity condition~\eqref{eq:regularitymutheta}.
}
\new{
\begin{lemma}\label{lem:regularitycheck}
Under the settings of Lemma~\ref{lem:uniqueminimumtheta}, suppose furthermore Assumption~\ref{assum:finitedim}, that $V(x) \geqslant C|x|^2 - C$ for all $x\in E$  for some $c,C>0$ (if $E=\R^d$) and that
\begin{equation}
\label{eq:formeRenergie}
\mathcal E \po \mu \pf = \mu(V) +  R \po \psi  \pf - \frac12 |\mu(\varphi)|^2 \,,
\end{equation}
for a convex $R\in\mathcal C^2(\mathbb H,\R)$ with $\na^2 R$ growing at most polynomially.  Introducing for $\theta,m\in\mathbb H^2$ the probability density
\begin{equation}
\label{eq:numtheta}
\nu_{\theta,m}(x) \propto \exp\po - V(x) - \na R(m) \cdot \psi(x) + \theta \cdot \psi(x) \pf, 
\end{equation}
then, for all $\theta \in \mathbb H$, $\bmu = \nu_{\theta,m_\theta}$. Moreover, $m_\theta$ is the unique solution of the fixed point problem
\[m = \nu_{\theta,m}(\psi)\,,\]
and $\theta \mapsto m_\theta$ is $\mathcal C^1$ with $\na_\theta m_\theta$  a positive self-adjoint operator smaller than $\mathrm{Cov}_{\bmu}(\psi)$. Finally, $\theta \mapsto \mathcal F(\bmu)$ and $\hat w$ are $\mathcal C^1$, with
\[\hat \omega(\theta) = R(m_\theta)   -\na R(m_\theta) \cdot m_\theta   + \ln \int _{E} \exp\po - V - \na R(m_\theta) \cdot \psi + \theta \cdot \psi \pf  \,. \]
\end{lemma}
}

\begin{proof}
Without loss of generality, we assume that the orthogonal of $\{ \psi(x),\ x\in E\}$  is $\{0\}$. Indeed, if it is not the case, we can replace $\psi$ by $\Pi_0\psi$ where $\Pi_0$ is the orthogonal projection onto this space, set $\C_0 = (\Pi_0 \C^2\Pi_0)^{1/2}$ and $\varphi_0 = \C_0 \psi$, and  the energy~\eqref{eq:formeRenergie} is unchanged. This additional condition implies that, for all $\theta$, $\Sigma^2_\theta := \mathrm{Cov}_{\bmu}(\psi)$ is a non-singular positive matrix.

\new{
With the notation~\eqref{eq:numtheta}, the self-consistency equation~\eqref{eq:selfconstitenttheta} becomes
\[\bmu = \nu_{\theta,m_\theta}\,,\qquad m_\theta=\bmu(\psi)\,. \]
In other words, introducing the function $G:\mathbb H^2 \rightarrow \mathbb H$ given by
\[G(\theta,m) =  m- \nu_{\theta,m}(\psi)\,, \]
then $m_\theta$ is characterised by the equation
\[G(\theta,m_\theta) = 0\,.\]
Indeed, any solution $m$ would be such that $\nu_{\theta,m}$ would solve the self-consistency equation~\eqref{eq:selfconstitenttheta}, and by uniqueness we would have $\nu_{\theta,m}=\bmu$ and then $m= \nu_{\theta,m}(\psi) = m_\theta$. Now, $G$ is $\mathcal C^1$ in both its variables, with
\[
\na_\theta G(\theta,m) = - \Sigma^2_\theta  \,,\qquad  \na_m G(\theta,m) = I + \Sigma^2_\theta \na^2 R(m) \,.
\]
The convexity of $R$ implies that  
\[I + \Sigma^2_\theta \na^2 R(m) = \Sigma_\theta \po  I + \Sigma_\theta \na^2 R(m) \Sigma_\theta\pf \Sigma_\theta^{-1}\,,\]
is  non-singular. By the implicit function theorem, we get that $m_\theta$ is $\mathcal C^1$, with
\[\na_\theta m_\theta= \po I + \Sigma^2_\theta \na^2 R(m_\theta ) \pf^{-1} \Sigma^2_\theta = \Sigma_\theta \po  I + \Sigma_\theta \na^2 R(m_\theta ) \Sigma_\theta\pf^{-1} \Sigma_\theta \,. \]
This shows that $\na_\theta m_\theta$ is self adjoint, positive and smaller than $\Sigma_\theta^2$ (i.e. $u\cdot \na_\theta m_\theta u \leqslant |\Sigma_\theta u|^2$ for all $u\in\mathbb H$). }

\new{Next, using the self-consistency equation,
\begin{align*}
\mathcal F(\bmu) & = \bmu(V) + R(m_\theta) - \frac12 |\C m_\theta|^2 + \int_{E} \bmu \ln \bmu  \\
& =  R(m_\theta) - \frac12 |\C m_\theta|^2   +(\theta -\na R(m_\theta)) \cdot m_\theta   + \ln \int _{E} \exp\po - V - \na R(m_\theta) \cdot \psi + \theta \cdot \psi \pf   \,.
\end{align*}
This gives the expression for $\hat \omega$, and the differentiability of $\hat \omega$ and $\mathcal F$ by the differentiability of $m_\theta$, the quadratic growth of $V$ and the polynomial growth of $\na^2 R$, which allows to differentiate under the integral sign for the last term. This concludes the proof.
}
\end{proof}

\subsection{The intermediary implications}\label{sec:proof-intermediary} 

\begin{proof}[Proof of Proposition~\ref{prop:intermediaryimplications}]
First, assume that $\mathcal F$ satisfies the PŁ inequality~\eqref{eq:PLF}. Applying it with $\rho = \bar\mu_\theta$ for some $\theta \in\mathbb H$ and using that
\[\mathcal I(\rho) = \int_{E} \left|\na \ln \mu - D\mathcal E(\mu,\cdot)\right|^2 \mu =  \int_{E} \left|\na \ln \mu - D\mathcal E_0(\mu,\cdot) - \rho(\varphi)\cdot \na\varphi\right|^2 \mu  \] 
(recall the definition~\eqref{eq:Etheta} of $\mathcal E_\theta$) gives, thanks to the self-consistency equation~\eqref{eq:selfconstitenttheta}, 
\begin{align*}
\hF(m_\theta) - \inf\hF & \leqslant  \frac{1}{2\lambda} \int_{E} \left|(\C^{-1} \theta - \C \bar \mu_\theta(\rp))\cdot \na \varphi\right|^2 \bar\mu_\theta \\
& \leqslant \frac{\|\na \varphi \|_\infty^2}{2\lambda}|\C^{-1} \theta-\C m_\theta|^2\\
&= \frac{\|\na \varphi \|_\infty^2}{2\lambda}\left|\C^{-1}\na \hF(m_\theta)\right|^2\,,
\end{align*}
where we used~\eqref{eq:naomega}.  Second, assume the coarse-grained PŁ inequality~\eqref{eq:PLFhat}. Thanks to~\eqref{eq:omegaegalite} and then~\eqref{eq:naomega}, for any $\theta\in\mathbb H$,  
\[
\omega(\theta) - \inf \omega = \hF(m_\theta) - \inf\hF + \frac12|\C m_\theta- \C^{-1} \theta|^2  \leqslant \po \frac{1}{2\lambda'}  + \frac12 \pf|\C \na \omega(\theta)|^2\,.
\]
Similarly, if $\theta$ satisfies~\eqref{eq:PLomega}, for all $\theta\in\mathbb H$,
\[0 \leqslant \hF(m_\theta) - \inf\hF  = \omega(\theta) - \inf \omega - \frac12|\C m_\theta-\C^{-1}\theta|^2 \leqslant \po \frac{1}{2\lambda''} - \frac12 \pf |\C \na \omega(\theta)|^2\,.\]
If $\omega$ is not constant, there exists $\theta$ such that $\na\omega(\theta)  \neq 0$, and thus the previous inequality implies that $\lambda''<1$, and~\eqref{eq:PLFhat} with $\frac{1}{\lambda'} = \frac{1}{\lambda''}-1$.
\end{proof}

\subsection{The non-linear PŁ inequality}\label{sec:proof-nonlinearPL}

\begin{proof}[Proof of Theorem~\ref{thm:unifLSI->PL}]\
Let $\mu_*$ be a critical point of $\mathcal F$, so that $\mu_* \propto \exp (-\frac{\delta\mathcal E}{\delta m}(\mu_*,\cdot))$. Using that $\na_{x_i} U_N(\bx) = D\mathcal E(\pi_{\bx},x_i)$,
\[
\mathcal I(\mu_*^{\otimes N}|\mu_\infty^N) =  N \int_{E^N} \left|D \mathcal E(\mu_*,x_1) - D \mathcal E(\pi_{\bx},x_1) \right|^2 \mu_*^{\otimes N}(\dd \bx)  = \underset{N\rightarrow\infty } o(N)
\]
thanks to~\eqref{eq:CVDE}. The uniform-in-$N$ LSI for $\mu_\infty^N$ implies that
\[0 \leqslant \frac1N  \mathcal H(\mu_*^{\otimes N}|\mu_\infty^N) = \mathcal H(\mu_*) + \int_{E^N} \mathcal E(\pi_{\bx}) \mu_*^{\otimes N} -  \frac1N \ln Z_N  \underset{N\rightarrow\infty } \longrightarrow 0\,.
\]
Since 
\[\int_{E^N} \mathcal E(\pi_{\bx}) \mu_*^{\otimes N} \underset{N\rightarrow\infty } \longrightarrow \mathcal E\po \mu_*\pf\,,\]
according to~\eqref{eq:CVE}, we deduce that
\begin{equation}
\label{loc:sgdfhdmlkZN}
\frac1N \ln Z_N   \underset{N\rightarrow\infty } \longrightarrow \mathcal F(\mu_*) \,. 
\end{equation}
Since $Z_N$ is independent from the critical point $\mu_*$, and since we can take $\mu_*$ to be a minimiser of $\mathcal F$, we get that all critical points are minimisers and that $\frac1N\ln Z_N\rightarrow \inf\mathcal F$ as $N\rightarrow \infty$. Moreover, the argument in~\eqref{loc:argumentuniqueness} shows that the minimiser is unique, and is the limit of $\mu_\infty^{N,1}$. 

As a consequence of~\eqref{loc:sgdfhdmlkZN}, for any $\mu \in \mathcal P_2(\R^d)$, by the law of large number, the lower semi-continuity of $\mathcal E$ and the Fatou lemma, 
\[\liminf_{N\rightarrow \infty} \frac1N  \mathcal H(\mu^{\otimes N}|\mu_\infty^N)  = \liminf_{N\rightarrow \infty} \co  \mathcal H(\mu) + \int_{E^N} \mathcal E(\pi_{\bx}) \mu^{\otimes N} -  \frac1N \ln Z_N    \cf \geqslant \mathcal F^c(\mu)\,.\] 
Moreover, for any $\mu\in\mathcal P_2(E)$ such that $\mathcal I(\mu) <\infty$,
\[\frac1N \mathcal I\po \mu^{\otimes N}|\mu_\infty^N \pf = \int_{E^N} \left|\na \ln \mu(x_1) - D \mathcal E(\pi_{\bx},x_1) \right|^2 \mu^{\otimes N}(\dd \bx) \,, \]
hence, writing $e_N $ the left-hand side of~\eqref{eq:CVDE}, 
\[\left|\frac1N \mathcal I\po \mu^{\otimes N}|\mu_\infty^N \pf - \mathcal I(\mu) \right| \leqslant 2\sqrt{e_N \mathcal I(\mu)} + e_N  \underset{N\rightarrow\infty } \longrightarrow 0\,. \]
Hence, dividing by $N$ the LSI~\eqref{eq:unifLSI} and taking the limsup as $N\rightarrow \infty$ gives the PŁ inequality~\eqref{eq:PLF} with $\lambda \geqslant \bar \lambda$,  which concludes the proof.
\end{proof}

\subsection{Convexity conditions}

\begin{proof}[Proof of Lemma~\ref{lem:convexite}]
Recall from Lemma~\ref{lem:omegaFhat} that  $\hF(m_\theta)  = \mathcal F(\bar\mu_\theta)$ with $  m_\theta = \bar\mu_\theta(\psi)$.  For $\theta,\zeta\in \mathbb H$ such that $m_\theta,m_\zeta \in \mathcal B$ and $t\in[0,1]$, 
\[t \bar \mu_\theta(\rp) + (1-t)\bar\mu_\zeta(\rp) = tm_\theta + (1-t)m_\zeta \in \mathcal B \,. \]
As a consequence, by definition~\eqref{eq:defFhat} of $\hF$,
\begin{align*}
\hF\po tm_\theta + (1-t)m_\zeta\pf & \leqslant \mathcal F \po t \bar \mu_\theta  + (1-t)\bar\mu_\zeta  \pf \\
& \leqslant t \mathcal F(\bar\mu_\theta) + (1-t) \mathcal F(\bar\mu_\zeta)  - \frac{\kappa}{2}t(1-t) |\bar \mu_\theta(\varphi)-\bar \mu_\zeta(\varepsilon)|^2 \\
&= t \hF(m_\theta) + (1-t) \hF(m_\zeta) - \frac{\kappa}{2}t(1-t)|\C(m_\theta - m_\zeta)|^2\,.
\end{align*}
For the second statement, we follow the proof of \cite[Lemma 4.3]{Dagallier}. Using that $\omega(\theta) = \frac12|\C^{-1}\theta|^2 + \hat \omega(\theta)$ with $\hat \omega$ given by~\eqref{eq:defomega1} (we can restrict to $\theta,\zeta \in \mathrm{Im}\C^{-1}$, otherwise $\omega(\theta)$ or $\omega(\zeta)$ is infinite and there is nothing to prove)
\begin{eqnarray*}
\lefteqn{ t \omega(\theta_1) + (1-t)\omega(\theta_0)}
\\
 &= &\inf_{m_0,m_1\in \mathbb H} \left\{ t \co \hF(m_1) + \frac12|\C m_1 - \C^{-1}\theta_1|^2  \cf + (1-t)\co \hF(m_0) + \frac12|\C m_0 - \C^{-1}\theta_0|^2  \cf \right\} \\
 & \geqslant & \inf_{m_0,m_1\in \mathbb H} \left\{ \hF(m_t) + \frac{\kappa}{2} t(1-t)|\C(m_1-m_0)|^2 +  \frac{t}2|\C m_1 - \C^{-1}\theta_1|^2   + \frac{1-t}2|\C m_0 - \C^{-1}\theta_0|^2    \right\}
\end{eqnarray*}
where $m_t = tm_1+(1-t) m_0$. Write  Instead of $m_0,m_1$, we can run the infimum over $m_t $ and $\Delta m =m_1-m_0$ (i.e. write $m_0 = m_t - t \Delta m$ 	and $m_1 = m_t+(1-t) \Delta m $). Writing $\theta_t = t\theta_1+(1-t)\theta_0$ and $\Delta \theta= \theta_1-\theta_0$ and rearranging  the squares give
\begin{eqnarray*}
\lefteqn{ t \omega(\theta_1) + (1-t)\omega(\theta_0)}
\\
& \geqslant & \inf_{m_t,\Delta m\in \mathbb H} \left\{ \hF(m_t) + \frac{\kappa}{2} t(1-t)|\C\Delta m|^2 +  \frac{1}2|\C m_t - \C^{-1}\theta_t|^2   + \frac{t(1-t)}2|\C\Delta  m - \C^{-1}\Delta \theta|^2    \right\}\\
& = & \omega(\theta_t) +   \frac{t(1-t)\kappa }{2(1+\kappa)}|\C^{-1} \Delta \theta|^2\,,
\end{eqnarray*}
as desired.
\end{proof}

\subsection{The degenerate case}

\begin{proof}[(Partial) proof of Theorem~\ref{thm:intermediaryimplications-dege}]
The proofs of items 1, 2 and 3 (resp. 5) are mutatis mutandis the same as in Proposition~\ref{prop:intermediaryimplications} (resp. Theorem~\ref{thm:unifLSI->PL}), hence we omit them. Let us simply mention that in item 3 we use that $\na \omega$ is neither constant nor Lipschitz to say that for all $r\geqslant 0$ there exists a $\theta\in\mathbb H$ with $|\C\na \omega(\theta)|=r$ (in fact, in view of~\eqref{eq:naomega}, it is easily seen that $\omega$ cannot be Lipschitz), which forces $\Phi_\omega(r)>r$.

Let us provide some details for item 4 for which, similarly, the arguments largely overlap those of Theorem~\ref{thm:PL->unifLSI}. In particular, Lemmas~\ref{lem:microscopicLSI}, \ref{lem:gammaX} and \ref{lem:omegaNtoomega} are still true (they only depend on Assumption~\ref{assum:PL->unifLSI}, not on the LSI satisfied by $\tilde\nu_N$). Starting from the decomposition~\ref{eq:micromacroentropie}, we can still control the microscopic term by~\eqref{loc:qfsffa}. The first difference appears at~\eqref{loc:vdfgemfa}, which becomes
\begin{equation}
\label{loc:vdfgemfa-dege}
\frac1N \int_{\mathbb H} F \ln F \dd \nu_N  \leqslant \widetilde{\Phi}\po \frac{4}{N^2} \int_{\mathbb H} |\C \na_\theta \sqrt{ F}|^2  \dd \nu_N \pf \,.
\end{equation}
Then we control $|\C \na_\theta \sqrt{F}|$ exactly as in the non-degenerate case with~\eqref{loc:dmlpaldfdg}, and then we follow the proof and use that $\widetilde{\Phi}$ is non-decreasing to end up with
\[\frac1N \int_{\mathbb H} F \ln F \dd \nu_N  \leqslant \widetilde{\Phi}\po \frac{\re{4}M\|\na \varphi\|_\infty^2 }{N} \mathcal I(\nu|\mu_\infty^N)\pf \,.
\]
Combining this with the microscopic part~\eqref{loc:qfsffa} concludes the proof of item 5.
\end{proof}

\subsection*{Acknowledgments}

The research of PM is supported by the project CONVIVIALITY (ANR-23-CE40-0003) of the
French National Research Agency. PM thanks Thierry Bodineau, Benoît Dagallier and Songbo Wang for fruitful discussions related to \cite{Dagallier,SongboLSI,Monmarchemetastable}.  The total amount of generative artificial
intelligence tools involved in this work is exactly zero.

\bibliographystyle{plain} 
\bibliography{biblio}

\end{document}